\documentclass[aop,MSNbibl,citesort,dvips]{arximspdf}
\usepackage{graphicx}

\doi{10.1214/11-AOP695}
\volume{41}
\issue{3A}
\pubyear{2013}
\firstpage{1513}
\lastpage{1555}

\makeatletter

\renewcommand{\epsilon}{\varepsilon}
\renewcommand{\mid}{|}

\newtheorem{theorem}{Theorem}
\newtheorem{corollary}{Corollary}[section]
\newtheorem{lemma}[corollary]{Lemma}
\newtheorem{proposition}[corollary]{Proposition}

\newproclaim{definition}{Definition}

\newproclaim{remark}{Remark}

\newcommand{\Prob}{{\mathbb P}}
\newcommand{\E}{{\mathbb E}}
\newcommand{\R}{{\mathbb R}}
\newcommand{\C}{{\mathbb C}}
\newcommand{\Half}{{\mathbb H}}
\newcommand{\Disk}{{\mathbb D}}

\newcommand{\F}{{\mathcal F}}
\newcommand{\seq}{{\mathcal S}}
\newcommand{\exc}{{\mathcal E}}
\newcommand{\imax}{{\mathcal I}}
\newcommand{\bd}{\partial}

\newcommand{\dist}{\operatorname{dist}}
\newcommand{\inrad}{\operatorname{inrad}}
\newcommand{\diam}{\operatorname{diam}}
\renewcommand{\Re}{\operatorname{Re}}
\renewcommand{\Im}{\operatorname{Im}}

\makeatother

\begin{document}
\begin{frontmatter}

\title{Multi-point Green's functions for SLE and an~estimate of Beffara}
\runtitle{Multi-point Green's functions for SLE}

\begin{aug}
\author[A]{\fnms{Gregory F.} \snm{Lawler}\thanksref{t1}\ead[label=e1]{lawler@math.uchicago.edu}}
\and
\author[A]{\fnms{Brent M.} \snm{Werness}\corref{}\ead[label=e2]{bwerness@math.uchicago.edu}}
\runauthor{G. F. Lawler and B. M. Werness}
\affiliation{University of Chicago}
\address[A]{Department of Mathematics\\
University of Chicago\\
5734 University Avenue\\
Chicago, Illinois 60637-1546 \\
USA\\
\printead{e1}\\
\phantom{E-mail: }\printead*{e2}} 
\end{aug}

\thankstext{t1}{Supported by NSF Grant DMS-09-07143.}

\received{\smonth{1} \syear{2011}}
\revised{\smonth{7} \syear{2011}}

%
\begin{abstract}
In this paper we define and prove of the existence of the multi-point
Green's function for $\mbox{SLE}$---a normalized limit of the probability that
an $\mbox{SLE}_\kappa$ curve passes near to a pair of marked points in the
interior of a domain. When $\kappa<8$ this probability is nontrivial,
and an expression can be written in terms two-sided radial $\mbox{SLE}$. One
of the main components to our proof is a refinement of a bound first
provided by Beffara [\textit{Ann. Probab.} \textbf{36} (2008)
1421--1452]. This work contains a proof of
this bound independent from the original.
\end{abstract}

%

\end{frontmatter}

\section{Introduction}\label{intro}

The Schramm--Loewner evolution ($\mbox{SLE}$) is a random process first
introduced by Oded Schramm in~\cite{first} as a candidate for scaling
limits of models from statistical physics which are believed to be
conformally invariant. Since its introduction, $\mbox{SLE}$ has been
rigorously established as the scaling limit for a number of these
processes, including the loop-erased random walk~\cite{loop}, the
percolation exploration process~\cite{perc} and the interface of the
Gaussian free field~\cite{gff}. For a general introduction to $\mbox{SLE}$
see, for example,~\cite{Lbook,parkcity,werner}.

Chordal $\mbox{SLE}_\kappa$ for $\kappa> 0$ in the upper half-plane ($\Half$)
is a one-parameter family of noncrossing random curves $\gamma\dvtx
[0,\infty) \rightarrow\overline\Half$ with $\gamma(0) = 0$ and\break
$\gamma
(\infty^-) = \infty$. Depending on $\kappa$, the geometry of the curve
has several different phases. When $0<\kappa\le4$, the curves are
simple (no self intersections). When $\kappa> 4$, the curves are no
longer simple, but they remain noncrossing. When $\kappa\ge8$, the
curve is space filling, passing through every point in $\overline
{\Half}$.

When examining geometric questions about the $\mbox{SLE}$ curves, such as the
almost sure Hausdorff dimension in~\cite{Beffara}, it is often useful
to be able to provide estimates on the probability that the process
$\gamma(t)$ passes near a series of marked points in $\Half$. However,
the non-Markovian nature of this process makes estimating such
probabilities difficult.\vadjust{\goodbreak}

When trying to understand the probability that $\mbox{SLE}_\kappa$ gets near
to some point $z \in\Half$ it is convenient to consider the conformal
radius of $z$ in $H_t := \Half\setminus\gamma(0,t]$, which we denote
by $\Upsilon_t(z)$, instead of the Euclidean distance from $z$ to
$\gamma(0,t]$; see Section~\ref{notations} for the definition. This
change does little to the geometry of the problem being considered
since the conformal radius differs from the Euclidean distance by at
most a universal multiplicative constant.

The Green's function for $\mbox{SLE}_\kappa$ from $0$ to $\infty$ in $\Half$ for
$\kappa< 8$ is a form of the normalized probability of passing near to
a point in $\Half$. It is defined by
\[
\lim_{\epsilon\rightarrow0} \epsilon^{d-2} \Prob\{\Upsilon_\infty(z)
< \epsilon\} = c_*G_\Half(z;0,\infty),
\]
where $d := 1+ \kappa/8$ is the Hausdorff dimension of the $\mbox{SLE}_\kappa
$, and $c_*$ is some known constant depending on $\kappa$. The Green's
function was first computed in~\cite{RS} (although they neither used
this name nor definition), and the exact formula found there is given
in Section~\ref{notations}. The existence of the limit requires some
argument, and a form of it is proven in Lemma~\ref{fastGreen}.

We wish to show analogously that
\[
\lim_{\epsilon,\delta\rightarrow0} \epsilon^{d-2} \delta^{d-2}
\Prob\{
\Upsilon_\infty(z)<\epsilon; \Upsilon_\infty(w)<\delta\}
\]
exists and can be written as
\[
c_*^2G_{\Half}(z;0,\infty)\E_z^*[G_{H_{T_z}}(w;z,\infty
)]+c_*^2G_{\Half
}(w;0,\infty)\E_w^*[G_{H_{T_w}}(z;w,\infty)],
\]
where $\E^*_z$ is the expectation of a particular form of $\mbox{SLE}$ called
\textit{two-sided radial $\mbox{SLE}$}, which can be understood as chordal $\mbox{SLE}$
conditioned to pass though the point $z$, and $G_{H_{T_z}}$\vspace*{1pt} is the
Green's function for $\mbox{SLE}$ in the domain remaining at the time it does
so. The form of the limit as the sum of two similar terms comes from
the two possible orders that the curve can pass through $z$ and $w$,
and each term individually can be thought of as an ordered Green's function.

To prove this result, we will use techniques similar to those used in~\cite{Beffara}, where Beffara (in slightly different notation)
established the estimate that there exists some $c > 0$ such that for
any two points $z,w \in\Half$ with $\Im(z),\break \Im(w)\ge1$
\[
\Prob\{\Upsilon_\infty(z) < \epsilon; \Upsilon_\infty(w) <
\epsilon\} <
c \epsilon^{2(2-d)}|z-w|^{d-2}.
\]

Similar techniques arise since both proofs need to make rigorous the
heuristic that an $\mbox{SLE}$ curve conditioned to pass through $z$ and then
$w$ will do so directly---without first approaching very near $w$
before passing through $z$.
Figure~\ref{goodvsevil} demonstrates some of the issues which can occur
which make this a tricky statement to make rigorous.

%
\begin{figure}

\includegraphics{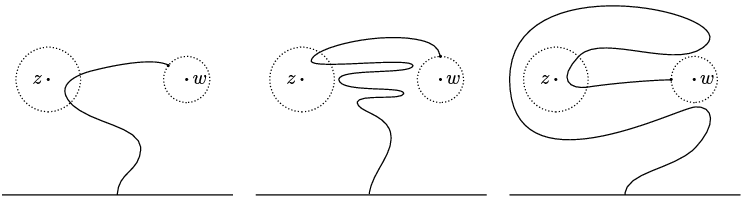}

\caption{We wish to show that curves that get near $z$ then near $w$
concentrate on curves like those in the left image. Estimating the
probability of such curves is easy by repeated application of the Green's
function. However, such simple estimation gives the same order of
magnitude to curves like those in the center image. This issue can be
overcome as long as getting near to $w$ before $z$ decreases the
probability that the $\mbox{SLE}$ gets even closer to $w$ later on. This is
often the case; however, the right image shows an example where it is
not. In this case, once the curve gets near to $z$, it is essentially
guaranteed to pass near $w$. Controlling for these issues forms the
bulk of this work.}
\label{goodvsevil}
\end{figure}

In the process of proving the existence of the multi-point Green's
function for $\mbox{SLE}$, we also obtain an independent proof of a mild
generalization of Beffara's estimate---that there exists a $c > 0$ such
that for any $z,w \in\Half$ with $\Im(z)$, \mbox{$\Im(w) \ge1$}
\[
\Prob\{\Upsilon_\infty(z) < \epsilon; \Upsilon_\infty(w) < \delta
\} < c
\epsilon^{2-d}\delta^{2-d}|z-w|^{d-2}.
\]
While it may be possible to derive some of the lemmas we require
directly from the proof in~\cite{Beffara}, we include a complete proof
of them, along with Beffara's original estimate, so that the proof of
our main result is completely self-contained.

It is worth noting that Beffara's estimate itself immediately yields an
upper bound on the multi-point Green's function. For a lower bound, and
an application of the multi-point Green's function to the proof of the
existence of the ``natural parametrization'' of $\mbox{SLE}$, a
parametrization of $\mbox{SLE}$ by what can be thought of as a form
$d$-dimensional arc length; see~\cite{nat}.

The paper is structured as follows. Section~\ref{notations} begins by
establishing the notation used throughout the paper, and to provide a
few simple deterministic and random bounds required in the proofs that
follow. Section~\ref{twosided} then gives a brief introduction to
two-sided radial $\mbox{SLE}$ and collects the facts about this process that
we require to show the existence of the multi-point Green's function.
Section~\ref{multipoint} provides a proof of the existence of the
multi-point Green's function assuming an estimate derived from our proof
of Beffara's estimate. The rest of the paper is dedicated to our
independent proof of Beffara's estimate. To aid in the presentation of
this proof, we have separated the bounds required by the type of
argument required: topological lemmas, probabilistic lemmas and
combinatorial lemmas. The proof of one of the topological lemmas is
left to the \hyperref[app]{Appendix} as the result is intuitive and the formal proof of
it does little to aid the understanding of our main
results.\looseness=-1

In this paper we fix $\kappa< 8$ and constants implicitly depend on
$\kappa$.

\section{Preliminaries}

\subsection{Notation}\label{notations}

We set
\begin{eqnarray*}
a &=& \frac2{\kappa},\qquad d = 1 + \frac\kappa8 = 1 + \frac1{4a}, \\
\beta&=&\frac8 \kappa- 1 = 4a-1> 0.
\end{eqnarray*}

The Green's function for chordal $\mbox{SLE}_\kappa$ (from $0$ to
$\infty$ in $\Half$) is
\[
G(x+iy) = G(r e^{i\theta})
= r^{d-2} \sin^{4a+{1}/({4a})-2} \theta=
y^{d-2} \sin^{\beta} \theta.
\]
The Green's function can be defined for other
simply connected domains as we now demonstrate.
If $D$ is a simply connected domain,
$z_1,z_2$ are distinct boundary points, let
$\Phi_D\dvtx D \rightarrow\Half$ be a conformal transformation
with \mbox{$\Phi_D(z_1) = 0$}, $\Phi_D(z_2) = \infty$. This is
unique up to a final dilation.
If
$w \in D$,
we define
\[
S_D(w;z_1,z_2) = \sin\arg\Phi_D(w) ,
\]
which is independent of the choice of $\Phi_D$ and
gives a conformal invariant. We let
$\Upsilon_D(w)$ be (twice the) conformal radius
of $w$ in $D$; that is, if
$f\dvtx \Disk\rightarrow D$ is a conformal transformation with
$f(0) = w$, then $\Upsilon_D(w) = 2 |f'(0)|$.
This satisfies the scaling rule
\[
\Upsilon_{f(D)}(f(w)) = |f'(w)| \Upsilon_D
(w).
\]
It is easy to check that $\Upsilon_\Half(x+iy) = y$, and, more
generally,
\[
\Upsilon_D(w) =
\frac{\Im(\Phi_D(w))}{|\Phi_D'(w)|}.
\]
The Green's function for $\mbox{SLE}_\kappa$ from $z_1$ to $z_2$
in $D$ is defined by
\[
G_D(w;z_1,z_2) = \Upsilon_D(w)^{d-2} S(w;z_1,z_2)^\beta.
\]
It satisfies the scaling rule
\[
G_D(w;z_1,z_2) = |f'(w)|^{2-d}
G_{f(D)}(f(w);f(z_1),f(z_2)) .
\]
For a proof that the Green's function so defined satisfies the limit
claimed in the \hyperref[intro]{Introduction}; see Lemma~\ref{fastGreen}.

Let $\inrad_D(w) = \dist(w,\partial D)$ denote the inradius.
Using the Koebe $(1/4)$-theorem, we know that
%
\begin{equation} \label{koebe}
\tfrac12 \inrad_D(w)
\leq\Upsilon_D(w)
\leq2 \inrad_D(w).
\end{equation}
Therefore,
\[
{G_D(w;z_1,z_2)} \asymp
{\inrad_D(w)^{d-2}
S_D(w;z_1,z_2)^{\beta} },
\]
where we write $f_1 \asymp f_2$ if there exists some constant $c$ such
that $f_1 \leq c f_2$ and
$f_2 \leq c f_1$.
We write
\[
\partial D = \partial_1 D \cup\partial_2 D \cup\{z_1,z_2 \},
\]
where $\partial_1 D, \partial_2 D$
denote the two open arcs of $\partial D$ with endpoints $z_1,z_2$. Let
$\hat S_D(w;z_1,z_2)$ be the minimum of the harmonic measures of
$\partial_1 D, \partial_2 D$ from~$w$. This is a conformal invariant,
and a
simple computation in $\Half$ shows
that
\[
\hat S_D(w;z_1,z_2) =\frac1 \pi\min\{\arg
\Phi_D(w), \pi-
\arg\Phi_D(w)\},
\]
and hence
\[
\hat S_D(w;z_1,z_2) \asymp S_D(w;z_1,z_2)
\]
and
\[
{G_D(w;z_1,z_2)} \asymp
{\inrad_D(w)^{d-2}
\hat S_D(w;z_1,z_2)^{\beta} }.
\]

To bound the harmonic measure, it is often useful to use the Beurling
estimate. We recall it here; for a proof see, for example,
\cite{probTech}, Chapter V. Let $B_t$ be a standard Brownian motion and
$\tau
_D$ denote the first exit time of some domain $D$ for this Brownian motion.
\begin{proposition}[(Beurling estimate)]\label{beurling1}
There is a constant $c > 0$ such that if $z \in\Disk$, and $K$ is a
connected compact subset of $\overline\Disk$ with $0 \in K$ and $K
\cap
\bd\Disk\neq\varnothing$, then
\[
\Prob^z\{B[0,\tau_\Disk]\cap K = \varnothing\} \le c |z|^{1/2}.
\]
\end{proposition}

We may derive from this the following consequence.
\begin{proposition}\label{beurling2}
There is a constant $c > 0$ such that if $K$ is a connected compact
subset of $\overline\Half$ with $K \cap\R\neq\varnothing$, and $z_0
\in\Half$, $\epsilon> 0$ are such that $B_\epsilon(z_0) \cap K \neq
\varnothing$ then for $w \in\Half$,
\[
\Prob^w\{B[0,\tau_{\Half\setminus K}] \cap B_{\epsilon}(z_0) \neq
\varnothing\} \le c \biggl[\frac{\epsilon}{|z_0-w|}\biggr]^{1/2}.
\]
\end{proposition}
\begin{pf}
Consider the map
\[
g(z) := \frac{\epsilon}{z-z_0},\qquad g \dvtx \C\setminus B_\epsilon(z_0)
\rightarrow\Disk.
\]
Let $K' = g([\C\setminus\Half]\cup[K\setminus B_\epsilon(z_0)])$,
and note that $K'$ is a connected compact subset of $\Disk$ with $0
\in
K'$ and $K' \cap\bd\Disk\neq\varnothing$. Thus by Proposition \ref
{beurling1} we know
\[
\Prob^{g(w)}\{B[0,\tau_\Disk]\cap K' = \varnothing\} \le c |g(w)|^{1/2},
\]
which, by the conformal invariance of Brownian motion and the
definition of $g$, is the desired statement.
\end{pf}

If $j=1,2$, let $\Delta_{D,j}(w;z_1,z_2)$ be the infimum of all $s$
such that there exists a curve $\eta\dvtx[0,1) \rightarrow D$ contained in
the disk of radius $s$ about
$w$ with $\eta(0) = w,\eta(1^-) \in\partial_jD$. Note that
\[
\inrad_D(w) = \min\{ \Delta_{D,1}(w;z_1,z_2),
\Delta_{D,2}(w;z_1,z_2)\}.
\]
We let
\[
\Delta_D^*(w;z_1,z_2) = \max\{ \Delta_{D,1}(w;z_1,z_2),
\Delta_{D,2}(w;z_1,z_2)\}.
\]
The Beurling estimate
implies that there is a $c < \infty$ such
that the probability a Brownian motion starting at $w$
reaches distance $\Delta_D^*(w;z_1,z_2)$ before leaving $D$
is bounded above by
\[
c\biggl[\frac{\inrad_D(w)}{\Delta_D^*(w;z_1,z_2)}\biggr]^{1/2}.
\]
Therefore,
%
\begin{equation} \label{beurling}
S_D(w;z_1,z_2) \asymp\hat S_D(w;z_1,z_2)
\leq c \biggl[\frac{\inrad_D(w)}{\Delta_D^*(w;z_1,z_2)} \biggr]^{1/2},
\end{equation}
which gives us the upper bound
\[
G_D(w;z_1,z_2) \le c \inrad_D(w)^{d-2+\beta/2}\Delta
_D^*(w;z_1,z_2)^{-\beta/2}.
\]

We will also need a fact which is a form of continuity of the Green's
function under a small perturbation of the domain. First consider the
following two lemmas on the conformal radius.
%
%
%
\begin{lemma} Let ${\mathcal B}_r$ denote the closed disk of radius
$e^{-r}$ about the
origin. Suppose $D$ is a simply connected subdomain of $\Disk$
containing the origin and
$e^{-r} < \inrad_D(0)$. Suppose $B_t $ is a Brownian motion starting at
the origin, and let
\[
\tau_D = \inf\{t\dvtx B_t \notin D\},\qquad
\tau_\Disk= \inf\{t \dvtx B_t \notin
\Disk\},\qquad
\sigma_{r,D} = \inf\{t \geq\tau_D\dvtx B_t \in{\mathcal B}_r\}.
\]
Then
\[
\Prob\{\tau_D < \sigma_{r,D} < \tau_\Disk\} = - \frac1r \log
[\Upsilon
_D(0)/2] .
\]
\end{lemma}
\begin{pf}
Let $f\dvtx D \rightarrow\Disk$
be the conformal transformation with $f'(0) > 0$. Since $\Upsilon_D(0)$
is twice the usual conformal radius,
$- \log[\Upsilon_D(0)/2] = \log f'(0)$.
Let $g(z) = \log[|f(z)|/|z|]$ which
is a bounded harmonic function on $D$, and hence
\[
\log f'(0) = g(0) =
\E[g(B_\tau)] = - \E[{\log}|B_\tau|].
\]
For $e^{-r} \leq|w| < 1$, $-{\log}|w|/r$ is the probability that a Brownian
motion starting at $w$ hits ${\mathcal B}$ about the origin before
leaving the $\Disk$. Therefore,
\[
\log f'(0) = r \Prob\{\tau_D < \sigma_{r,D} < \tau_\Disk\}.
\]
\upqed\end{pf}
\begin{lemma}\label{BeurlingConformal}
There exists a $c>0$ such that for any two simply connected domains
$D_1 \subseteq D_2$ and a point $w \in D_1 \cap D_2$, then
\[
0 \le\Upsilon_{D_2}(w) - \Upsilon_{D_1}(w) \le c \diam(D_2
\setminus D_1).
\]
\end{lemma}
\begin{pf} Without loss of generality, we may assume $\inrad(D_2) = 1$.
If $\inrad(D_1) \leq7/8$, then $\diam(D_2 \setminus D_1)
\geq1/8$, and
we can use the estimate $\inrad(D) \asymp\Upsilon(D)$. If $\inrad
(D_1) \geq7/8$, then
we can use the previous lemma, conformal invariance, and
the Koebe $(1/4)$-theorem to see $\Upsilon_{D_2}(w)
- \Upsilon_{D_1}(w)$ is comparable to the probability that a Brownian
motion starting at
$w$ hits $D_2 \setminus D_1$ and returns to ${\mathcal B} = B_{1/16}(w)$,
the disk of radius $1/16$ about $w$ without
leaving $D_2$. Using the Beurling estimate, we see the probability of hitting
$D_2 \setminus D_1$ is bounded above by $c \diam(D_2\setminus
D_1)^{1/2}$ and
using it again the probability of getting back to ${\mathcal B}$ before
leaving $D_2$ is
bounded by $c \diam(D_2\setminus D_1)^{1/2}$.
\end{pf}

We will need some notion of closeness of two nested domains before we
can state our lemma. Although the following definitions are very
general, we will use them only in the case where the domains are the
complements of a single curve considered up to two different times.
\begin{definition*}
Given two nested simply connected domains $D_1 \subseteq D_2 \subseteq
\Half$ with marked boundary points $z_1 \in\bd D_1$ and $z_2 \in\bd
D_2$, we say $(D_1,z_1)$ and $(D_2,z_2)$ are \textit{$R$-close near $z$}
if the following holds. Let $B_R^{(i)}(z)$ denote the connected
component of $B_R(z) \cap D_i$ which contains $z$. Then:
\begin{itemize}
\item$z_1 \in\bd B_R^{(1)}(z)$,
\item$z_2 \in\bd B_R^{(2)}(z)$ and
\item$D_2 \setminus D_1 \subseteq B_R(z)$.
\end{itemize}
\end{definition*}
%
%
\begin{lemma}\label{GreensLemma}
There exists $c > 0$ such that the following holds. Suppose $z,w \in
\Half$, $D_1 \subseteq D_2 \subseteq\Half$ are simply connected
domains, and $z_1 \in\bd D_1$, $z_2 \in\bd D_2$. If:
\begin{itemize}
\item$z,w \in D_1 \cap D_2$,
\item$(D_1,z_1)$ and $(D_2, z_2)$ are $R$-close near $z$ for $R \le
\inrad_{D_1}(w) \wedge\frac{1}{2}|z-w|$,
\item$\infty\in\bd D_1 \cap\bd D_2$,\vadjust{\goodbreak}
\end{itemize}
then
\[
|G_{D_1}(w;z_1,\infty) - G_{D_2}(w;z_2,\infty)| \le c \inrad
_{D_1}(w)^{d-2-({\beta\wedge1})/{2}} R^{({\beta\wedge1})/{2}}.
\]
\end{lemma}

One need not fix the point $z$ in the beginning of this lemma by simply
making the second bullet point of this lemma say that there exists some
$z$ so that the domains are $R$-close near $z$; however we write it in
this form since we will always use this lemma with a fixed $z$ and $w$
already in mind.
\begin{pf*}{Proof of Lemma~\ref{GreensLemma}}
Recall that
\[
G_D(w;z_1,z_2) = \Upsilon_D(w)^{d-2} S_D(w;z_1,z_2)^\beta,
\]
where $S(w;z_1,z_2)$ is the sine of the argument of $w$ after applying
the unique (up to scaling) conformal map, $\Phi_D$, that sends $D$ to
$\Half$ while sending $z_1$ to $0$ and $z_2$ to~$\infty$. Writing, as before,
\[
\bd D = \bd_1 D \cup\{z_1\} \cup\bd_2 D \cup\{z_2\},
\]
where the union is written in counter-clockwise order, this argument is
conformally invariant and can be computed by
\[
\arg\Phi_D(w) = \pi\cdot\Prob^w\{ B_\tau\in\bd_2 D\} \qquad\mbox{where }
\tau
= \inf\{t \dvtx B_t \in\bd D\},
\]
where $\Prob^w$ is the probability for a standard Brownian motion
started at $w$.

Consider our case. Write
\[
\bd D_1 = \bd_1 D_1 \cup\{z_1\} \cup\bd_2 D_1 \cup\{\infty\}
\quad\mbox{and}\quad
\bd D_2 = \bd_1 D_2 \cup\{z_2\} \cup\bd_2 D_2 \cup\{\infty\}
\]
again with the union written in counter-clockwise order. Note that the
condition that $(D_1,z_1)$ and $(D_2, z_2)$ are $R$-close near $z$
implies that
%
\begin{equation}\label{bdryChange}\quad
\bd_1 D_1 \setminus B_R(z) = \bd_1 D_2 \setminus B_R(z)
\quad\mbox{and}\quad
\bd_2 D_1
\setminus B_R(z) = \bd_2 D_2 \setminus B_R(z).
\end{equation}
Define
\[
\tau_1 = \inf\{t \dvtx B_t \in\bd D_1\} \quad\mbox{and}\quad \tau_2 = \inf\{t \dvtx B_t
\in
\bd D_2\}
\]
and note that, since $B_0 = w$, $\tau_1 \le\tau_2$.

We may write that
\begin{eqnarray*}
| {\arg\Phi_{D_1}(w)} - \arg\Phi_{D_2}(w)| & = & |\pi\cdot\Prob^w\{
B_{\tau_1} \in\bd_2 D_1\} - \pi\cdot\Prob^w\{ B_{\tau_2} \in\bd_2
D_2\}|\\
& \le & 2\pi\cdot\Prob^w\{B_t \in B_R(z) \mbox{ for some } t \le
\tau
_2\},
\end{eqnarray*}
where the last line follows since, if considered path-wise, the
Brownian motion must enter $B_R(z)$ if it is to hit a different side of
the boundary in $D_1$ versus $D_2$ by~(\ref{bdryChange}). By
the Beurling estimate (Proposition~\ref{beurling2}),
\[
| {\arg\Phi_{D_1}(w)} - \arg\Phi_{D_2}(w)| \le c \biggl(\frac{R}{|z-w|}\biggr)^{1/2}.
\]
By noting that $\inrad_{D_1}(w) \le c |z-w|$ by the choice of $R$ and
the definition of $R$-close, and splitting into the cases when $\beta
\ge1$ versus $\beta< 1$ we see
\[
| S_{D_1}(w;z_1,\infty)^\beta- S_{D_2}(w;z_2,\infty)^\beta| \le c
\biggl(\frac{R}{\inrad_{D_1}(w)}\biggr)^{(\beta\wedge1)/2}.
\]

Consider the term involving the conformal radius. By using Lemma \ref
{BeurlingConformal} and recalling that $d-2 < 0$ and $\Upsilon_{D_1}(w)
\le\Upsilon_{D_2}(w)$, we see
\begin{eqnarray*}
| \Upsilon_{D_2}(w)^{d-2} - \Upsilon_{D_1}(w)^{d-2} | & \le & (2-d)
\Upsilon_{D_1}(w)^{d-3} |\Upsilon_{D_2}(w) - \Upsilon_{D_1}(w)| \\
& \le & c \Upsilon_{D_1}(w)^{d-2} \biggl(\frac{R}{\inrad_{D_1}(w)}\biggr).
\end{eqnarray*}

Combining these, noting that $R < \inrad_{D_1}(w)$, gives
\begin{eqnarray*}
&&|G_{D_1}(w;z_1,\infty) - G_{D_2}(w;z_2,\infty)| \\
&&\qquad \le |\Upsilon_{D_1}(w)^{d-2}S_{D_1}(w;z_1,\infty)^\beta- \Upsilon
_{D_1}(w)^{d-2}S_{D_2}(w;z_2,\infty)^\beta| \\
&&\qquad\quad{} + |\Upsilon_{D_1}(w)^{d-2}S_{D_2}(w;z_2,\infty)^\beta-\Upsilon
_{D_2}(w)^{d-2}S_{D_2}(w;z_2,\infty)^\beta| \\
&&\qquad \le c \Upsilon_{D_1}(w)^{d-2} \biggl(\frac{R}{\inrad_{D_1}(w)}\biggr)^{(\beta
\wedge1)/2} + c \Upsilon_{D_1}(w)^{d-2} \biggl(\frac{R}{\inrad_{D_1}(w)}\biggr)
\\
&&\qquad \le c \inrad_{D_1}(w)^{d-2-({\beta\wedge1})/{2}} R^{
({\beta
\wedge1})/{2}}
\end{eqnarray*}
as desired.
\end{pf*}


\subsection{Schramm--Loewner evolution}\label{sle}

The chordal Schramm--Loewner evolution with parameter
$\kappa$ (from $0$ to $\infty$
in $\Half$ parametrized so that the half-plane capacity
grows at rate $a=2/\kappa$)
is the random curve $\gamma\dvtx[0,\infty) \rightarrow\overline
\Half$ with $\gamma(0) = 0$ satisfying the following. Let $H_t$
denote the unbounded component of $\Half\setminus\gamma(0,t]$,
and let $g_t$ be the unique conformal transformation of $H_t$
onto $\Half$ with $g_t(z) - z \rightarrow0$ as $z \rightarrow
\infty$. Then $g_t$ satisfies the Loewner differential
equation
%
\begin{equation} \label{chordaleq}
\partial_tg_t(z) = \frac{a}{g_t(z) - U_t},\qquad
g_0(z) = z ,
\end{equation}
where $U_t = -B_t$ is a standard Brownian motion. For
$z \in\overline\Half\setminus\{0\}$, the solution of this
initial value problem exists up to time $T_z \in(0,\infty]$.

Suppose $z \in\Half$, and let
\[
Z_t = Z_t(z) = X_t + i Y_t = g_t(z) - U_t .
\]
Then the Loewner differential equation becomes the
SDE
%
\begin{equation} \label{ZtDef}
dZ_t = \frac{a}{Z_t} \,dt + d B_t .
\end{equation}
Let
\begin{eqnarray*}
S_t &=& S_t(z) = S_{H_t}(z;\gamma(t),\infty) =
\sin\arg Z_t,
\\
\Upsilon_t &=& \Upsilon_t(z) = \Upsilon_{H_t}(z;\gamma(t),\infty)
= \frac{Y_t}{|g_t'(z)|},
\\
M_t &=& M_t(z) = G_{H_t}(z;\gamma(t),\infty)
= \Upsilon_t^{d-2} S_t^{\beta}.
\end{eqnarray*}
Either by direct computation or by using the Schwarz lemma, we
can see that $\Upsilon_t$ decreases in $t$, and hence we can define
$\Upsilon= \Upsilon_{T_z-}$. If $0 < \kappa\leq4$, the $\mbox{SLE}$ paths
are simple and with probability one $T_z = \infty$. If $4 < \kappa< 8$,
$T_z < \infty$ and by
(\ref{koebe}) we know
%
\begin{equation} \label{koebe2}
\Upsilon\asymp\dist\bigl[z,\gamma(0,T_z] \cup\R\bigr]
= \dist[z,\gamma(0,\infty) \cup\R].
\end{equation}
Using It\^o's formula, we can see that $M_t$ is a local
martingale satisfying
\[
dM_t = \frac{a X_t} {X_t^2 + Y_t^2}
M_t \,dB_t.
\]

We will need the following estimate for $\mbox{SLE}$; see~\cite{AK}
for a proof. By a \textit{crosscut} in $D$ we will mean a simple curve
$\eta\dvtx(0,1) \rightarrow D$ with $\eta(0^+), \eta(1^-) \in
\partial D$. We call $\eta(0^+),\eta(1^-)$ the endpoints of the crosscut.
\begin{proposition} \label{julprop}
There exists $c < \infty$ such that if
$\eta$ is a crosscut in $\Half$ with $-\infty<
\eta(1^-) \leq\eta(0^+)=-1$,
then the probability that an $\mbox{SLE}_\kappa$ curve
from $0$ to $\infty$ intersects $\eta$ is bounded above by
$c \diam(\eta)^\beta$ where $\beta= 4a-1$
is as defined in Section~\ref{notations}.
\end{proposition}

\subsection{Radial parametrization}\label{radial}

In order to prove the existence of multi-point Green's functions, we
will need to study the behavior of the $\mbox{SLE}$ curve from the perspective
of $z \in\Half$. To do so, it is useful to parametrize
the curve so that the conformal radius seen from $z$
decays deterministically. We fix $z \in\Half$ and let
\[
\sigma(t) = \inf\{s\dvtx \Upsilon_s = e^{-2at}
\}.
\]
Under this\vspace*{1pt} parametrization, the ``starting time'' is $-\log(\Upsilon
_0)/2a$, and the total lifetime of the
curve is $\log(\Upsilon_0/\Upsilon)/2a$.
Let $\Theta_t = \arg Z_{\sigma(t)}(z),
\hat S_t = S_{\sigma(t)}(z) = \sin\Theta_t$.
Using It\^o's formula
one can see that $\Theta_t$ satisfies
\[
d \Theta_t = (1-2a) \cot\Theta_t
\,dt + d\hat W_t,
\]
where $\hat W_t$ is a standard Brownian motion. Since $a > 1/4$,
comparison to a Bessel process shows that solutions to
this process leave $(0,\pi)$ in finite time. This reflects
that fact that chordal $\mbox{SLE}_\kappa$ does not reach $z$ for
$\kappa< 8$ and hence $\Upsilon>0$.
Let
\[
\hat M_t = M_{\sigma(t)}(z) = e^{-2at(d-2)}
\hat S_t^\beta= e^{-(2a - 1/2)t}
\hat S_t^\beta.
\]
This is a time change of a local martingale and hence is
a local martingale; indeed, It\^o's formula gives
\[
d\hat M_t = (4a-1) \cot\Theta_t
\,d\hat W_t.
\]
Using Girsanov's theorem (see, e.g.,~\cite{karatzas}), we can define a
new probability
measure $\Prob^*$ which corresponds to paths ``weighted
locally by the local martingale $\hat M_t$.'' For the time being, we
treat this as an arbitrary change of measure; however, in
Section~\ref{twosided} we will see that is precisely the change of
measure which gives two-sided radial $\mbox{SLE}$.
Intuitively, $\hat M_t$ weights more heavily those paths
whose continuations are likely to get much closer to
$z$. For more examples of the application of Girsanov's theorem to the
study of $\mbox{SLE}$, and a general outline of the way Girsanov's theorem is
used below, see~\cite{Buzios}.\looseness=-1

In this weighting,
\[
d\hat W_t =
(4a-1) \cot\Theta_t \,dt + dW_t ,
\]
where $W_t$ is a standard Brownian motion with respect to $\Prob^*$.
In particular,
%
\begin{equation} \label{sde1}
d\Theta_t = 2a \cot\Theta_t \,dt + dW_t.
\end{equation}
Since $2a > 1/2$, we can\vspace*{1pt} see by comparison with a Bessel process that
with respect to
$ \Prob^*$, the process stays in $(0,\pi)$ for all
times. Using this we can show that $ \hat M_t$ is actually
a martingale, and the measure $ \Prob^*$ can be defined by\looseness=-1
\[
\Prob^*[V] = \hat M_0^{-1}
\E[\hat M_t 1_V ] \qquad\mbox{for }
V \in\F_t,
\]\looseness=0
where $\F_t$ denotes the $\sigma$-algebra generated by
$\{\hat W_s\dvtx 0 \leq s \leq t \}$.
Much of the analysis of $\mbox{SLE}_\kappa$ as it gets close to
$z$ uses properties of the simple SDE (\ref{sde1}).
Recall that we assume that $a > 1/4$ and all
constants can depend on $a$.
\begin{lemma} \label{mar3lemma1}
There exists $c < \infty$ such that if
$\Theta_t$ satisfies (\ref{sde1}) with $\Theta_0
= x \in(0,\pi/2)$, then if $0 < y < 1$ and
\[
\tau= \inf\bigl\{t\dvtx \Theta_t \in\{y,\pi/2\}
\bigr\},
\]
then
\[
\Prob^*\{\Theta_\tau= y\} \leq c (y/x)^{1-4a} .
\]
\end{lemma}
\begin{pf} It\^o's formula shows that $F(\Theta_{
t \wedge\tau})$ is a $\Prob^*$-martingale where
\[
F(s) = \int_s^{\pi/2} (\sin u)^{-4a} \,du,\qquad
\frac{F''(s)}{F'(s)} = -4a \cot s.
\]
Note that $F(\pi/2)
= 0$ and
\[
F(s) \sim
\frac{s^{1-4a}}{1-4a} ,\qquad
s \rightarrow0^+ .
\]
The optional sampling theorem implies that
\[
F(x) = \Prob^*\{\Theta_\tau= y\} F(y).
\]
\upqed\end{pf}
\begin{lemma} The invariant density for the SDE (\ref{sde1})
is
%
\begin{equation}
f(x) = C_{4a} \sin^{4a} x ,\qquad 0 < x < \pi,\qquad
C_{4a} := \biggl[\int_0^\pi\sin^{4a} x
\biggr]^{-1}.\vspace*{-2pt}
\end{equation}
\end{lemma}
\begin{pf} This can be quickly verified and
is left to the reader.\vspace*{-2pt}
\end{pf}

One can use standard techniques for one-dimensional
diffusions to discuss the rate of convergence to the
equilibrium distribution. We will state
the one result that we need; see~\cite{nat} for more details.
If $F$ is a nonnegative function on $(0,\pi)$, let
\[
I_F:= C_{4a}
\int_0^{\pi} F(x) \sin^{4a} x \,dx .\vspace*{-2pt}
\]

\begin{lemma}\label{invdense} There exists $u < \infty$ such that for
every $t_0 > 0$ there exists $c < \infty$ such that if $F$ is
a nonnegative function with $I_F < \infty$ and
$t \geq t_0$,
\[
| \E[F(\Theta_t)] - I_F|
\leq c e^{-ut} I_F.\vspace*{-2pt}
\]
\end{lemma}

Note that this estimate applies uniformly over all starting points $x$.

An important case for us is $F(x) = [\sin x]^{-\beta}
= \sin^{1-4a} x$. Let
%
\begin{equation} \label{mar41}
c_* = I_F = \frac{C_{4a}}{C_{1}} = \frac{2}{\int_0^\pi\sin^{4a} x \,dx}.
\end{equation}

We will take advantage of this uniform bound to
give a concrete estimate on how well the Green's function approximates
the probability of getting near a point.\vspace*{-2pt}
\begin{lemma}\label{fastGreen} There exists $u > 0$
such that if $D$ is a simply connected domain, and $z_1,z_2$ are
points in its boundary, $r \leq3/4$, $\gamma$
is an $\mbox{SLE}_\kappa$ curve from $z_1$ to $z_2$, $w \in D$,
and $D_\infty$ denotes the connected component of
$D \setminus\gamma(0,\infty)$ containing $w$, then
\begin{eqnarray*}
\Prob\{\Upsilon_{D_\infty}(w) \le
r \cdot\Upsilon_D(w) \} & = & c_* r^{2-d}
S_D(w;z_1,z_2)^\beta[1 + O(r^u)] \\[-2pt]
& = & c_* r^{2-d} \Upsilon_D(w)^{2-d}
G_D(w;z_1,z_2) [1 + O(r^u)],
\end{eqnarray*}
where $c_*$ is as defined in (\ref{mar41}). In particular,
there exists $c < \infty$ such that for all $r \le3/4$,
\[
\Prob\{\Upsilon_{D_\infty}(w)
\leq
r \cdot\Upsilon_D(w) \} \leq
c r^{2-d} S_D(w;z_1,z_2)^\beta.\vspace*{-2pt}
\]
\end{lemma}
\begin{pf} By conformal invariance we may assume
$\Upsilon_D(w)=1 $ and define $t$ by $r=e^{-2at}$. Let
$\sigma= \inf\{s\dvtx \Upsilon_s = r \}$. Then,
\begin{eqnarray*}
\Prob\{ \sigma< \infty\} & = & \E[1\{\sigma< \infty\}]\\[-2pt]
& = & r^{2-d} \E[\hat M_t \hat S_t^{-\beta} ] \\
& = & r^{2-d} S_D(w;z_1,z_2)^\beta\E^*[\hat S_t^{-\beta} ]\\
& = & c_* r^{2-d} S_D(w;z_1,z_2)^\beta[1 + O(e^{-ut})]\\
& = & c_* r^{2-d} \Upsilon_D(w)^{2-d}
G_D(w;z_1,z_2) [1 + O(e^{-ut})].
\end{eqnarray*}
\upqed\end{pf}

Using (\ref{koebe}) and (\ref{beurling}), we immediately get
the following lemma which is in the form that we will use.
\begin{lemma}\label{july173} There exists $C < \infty$, such that if
$D$ is a simply connected domain, and $z_1,z_2$ are points in its
boundary, $r \leq3/4$, and $\gamma$
is an $\mbox{SLE}_\kappa$ curve from $z_1$ to $z_2$, then
\[
\Prob\{\dist[w,\gamma[0,\infty)] \leq
r \cdot\inrad_D(w) \} \leq C r^{2-d}
\biggl[\frac{\inrad_D(w)}{\Delta_D^*(w;z_1,z_2)} \biggr]
^{\beta/2}.
\]
\end{lemma}

\subsection{Two-sided radial $\mbox{SLE}$}\label{twosided}

We call $\mbox{SLE}_\kappa$ under the measure $\Prob^*$ in the
previous subsection \textit{two-sided radial $\mbox{SLE}_\kappa$}
(from $0$ to $\infty$ through $z$ in $\Half$ stopped
when it reaches $z$). Roughly speaking it is chordal
$\mbox{SLE}_\kappa$ conditioned to go through $z$ (stopped
when it reaches $z$). Of course this is an event of
probability zero, so we cannot define the process exactly this
way. We may provide a direct definition by driving the Loewner equation
by the process defined in (\ref{sde1}) rather than a standard
Brownian motion. This definition uses the radial parametrization. We
could also just as well use the
capacity parametrization, in which case with probability
one $T_z < \infty$.

One may justify the definition above examining its relationship to
$\mbox{SLE}_\kappa$
conditioned to get close to $z$. This next proposition is just a restatement
of the definition of the measure $\Prob^*$ when restricted to curves
stopped at a particular stopping time.
\begin{proposition}\label{RNderiv} Suppose $\gamma$ is a chordal
$\mbox{SLE}_\kappa$
path from $0$ to $\infty$ and $z \in\Half$. For $\epsilon
\leq\Im(z)$, let $\rho_\epsilon
= \inf\{t\dvtx \Upsilon_t(z) = \epsilon\}$. Let $\mu,\mu^*$
be the two measures on $\{\gamma(t)\dvtx 0 \leq t \leq
\rho_\epsilon\}$
corresponding to chordal $\mbox{SLE}_\kappa$ restricted to the
event $\{ \rho_\epsilon
< \infty\}$ and two-sided radial $\mbox{SLE}_\kappa$ through $z$,
respectively. Then
$\mu,\mu^*$ are mutually absolutely continuous with respect
to each other with the Radon--Nikodym derivative
\[
\frac{d\mu^*}{d\mu} = \frac{G_{H_{\rho_\epsilon}}(z;\gamma(\rho
_\epsilon
),\infty)}{G_\Half(z;0,\infty)} = \frac{\epsilon^{d-2} S_{\rho
_\epsilon
}(z)^\beta}{G_\Half(z;0,\infty)}.
\]
\end{proposition}

Note that as $\epsilon\rightarrow0$ the Radon--Nikodym derivative
tends to infinity. This reflects the fact that $\mu^*$ is a probability
measure and that the total mass of $\mu$ is of order $\epsilon^{2-d}$
(see Lemma~\ref{fastGreen}).

This proposition seems to indicate that there is a still
a significant difference between two-sided radial $\mbox{SLE}_\kappa$
going though $z$ and $\mbox{SLE}_\kappa$ conditioned to get within a
specific distance. However, by using the methods of Lemma~\ref
{invdense} we get the following improvement.
\begin{proposition}\label{measurelimit} There exists $u > 0,
c < \infty$ such that
the following is true.
Suppose $\gamma$ is a chordal $\mbox{SLE}_\kappa$
path from $0$ to $\infty$ and $z \in\Half$. For $\epsilon
\leq\Im(z)$, let $\rho_\epsilon
= \inf\{t\dvtx \Upsilon_t(z) = \epsilon\}$. Suppose $\epsilon' <
3\epsilon
/4$. Let $\mu',\mu^*$
be the two probability
measures on $\{\gamma(t)\dvtx 0 \leq t \leq
\rho_\epsilon\}$
corresponding to chordal $\mbox{SLE}_\kappa$ conditioned on the
event $\{ \rho_{\epsilon'}
< \infty\}$ and two-sided radial $\mbox{SLE}_\kappa$ through $z$,
respectively. Then
$\mu',\mu^*$ are mutually absolutely continuous with respect
to each other and the Radon--Nikodym derivative
satisfies
\[
\biggl|\frac{d\mu^*}{d\mu'}- 1 \biggr| \leq
c (\epsilon'/\epsilon)^u .
\]
\end{proposition}

From the definition, it is easy to show that there is a subsequence
$t_n \uparrow T_z$ with $\gamma(t_n) \rightarrow z$.
In fact, in~\cite{Lnew}, a stronger fact is proven: for $0 < k < 8$, with
probability one, the two-sided radial measure produces a curve, by which
we mean that with probability one $\gamma(T_z^-) = z$.
\begin{lemma}\label{approachlem}
Let $\rho_\epsilon= \inf\{t\dvtx \Upsilon_t(z) = \epsilon\}$. There
exists $\alpha> 0$ so that for any $z\in\Half$ there exists $c_z <
\infty$, so that for any $\epsilon$ and $R$ with $\epsilon\le R \le
\Im
(z)$ we have
\[
\Prob^*\{\gamma[\rho_\epsilon,T_z] \not\subseteq B_R(z) \} \le c_z
\biggl(\frac{\epsilon}{R}\biggr)^\alpha.
\]
\end{lemma}
\begin{pf}
This result was shown for a two-sided radial through $0$ from $1$ to
$-1$ in $\Disk$ in~\cite{Lnew}, Theorem 3. Since $c_z$ is allowed to
depend on $z$, the form in this lemma can be obtained by conformal invariance.
\end{pf}

We will also need this bound in a chordal form, rather than two-sided
radial form. In order to prove the chordal form, we need the following lemma.

\begin{lemma}\label{sinbound}
Let $\rho_\epsilon= \inf\{t\dvtx \Upsilon_t(z) = \epsilon\}$. There
exists $c <\infty$, such that if $z \in\Half$ and $\epsilon\leq\Im
(z)/2$, $0 < \theta_0
\leq\pi/2$,
\[
\Prob\{S_{\rho_\epsilon}(z) < \sin(\theta_0) | \rho_\epsilon<
\infty\}
\le c \theta_0^{2}.
\]
\end{lemma}
\begin{pf}
First note that by Proposition~\ref{RNderiv} and Lemma~\ref{fastGreen}
we have that
\[
\Prob\{S_{\rho_\epsilon}(z) < \sin(\theta_0) | {\rho_\epsilon} <
\infty
\} \leq c \E^*[S_{\rho_\epsilon}^{-\beta}(z)1\{S_{\rho_\epsilon
}(z) <
\sin(\theta_0)\}].
\]
By applying the techniques from Lemma~\ref{invdense} with the function
\[
F(\theta) = \sin(\theta)^{-\beta}1\{\sin(\theta) < \sin(\theta
_0)\},\vadjust{\goodbreak}
\]
and noting that the integral is
\[
\int_0^\pi\sin(\theta)^{-\beta}1\{\sin(\theta) < \sin(\theta
_0)\} \sin
^{4a} \,d\theta= 2\int_0^{\theta_0} \sin(\theta) \,d\theta= O(\theta_0^2),
\]
we get the result.
\end{pf}
\begin{lemma}\label{approachlemcond}
Let $\rho_\epsilon= \inf\{t \dvtx \Upsilon_t(z) = \epsilon\}$. Fix
$\epsilon< \eta< R < 1$ and $z \in\Half$, then there exists some $c$
depending only on $z$ and $\alpha> 0$ such that
\[
\Prob\{\gamma[\rho_{\eta},\rho_\epsilon] \not\subseteq B_R(z) |
{\rho
_\epsilon} < \infty\} \le c \biggl(\frac{\eta}{R}\biggr)^\alpha.
\]
\end{lemma}
\begin{pf}
Let $0 < \theta< \pi/2$ be arbitrary; we will fix its precise value
later. We apply Lemmas~\ref{sinbound} and~\ref{approachlem} with
the above to see that
\begin{eqnarray*}
&&\Prob\{\gamma[{\rho_\eta},{\rho_\epsilon}] \not\subseteq B_R(z)
| {\rho
_\epsilon}<\infty\} \\
&&\qquad = \Prob\{\gamma[{\rho_\eta},{\rho_\epsilon}] \not\subseteq
B_R(z) ;
S_{\rho_\epsilon}(z) \ge\sin(\theta) | {\rho_\epsilon}<\infty\}
\\
&&\qquad\quad{} + \Prob\{\gamma[{\rho_\eta},{\rho_\epsilon}] \not\subseteq
B_R(z) ;
S_{\rho_\epsilon}(z) < \sin(\theta) | {\rho_\epsilon}<\infty\} \\
&&\qquad \le c\E^*\bigl[S_{\rho_\epsilon}^{-\beta}(z)1\{\gamma[{\rho_\eta
},{\rho
_\epsilon}] \not\subseteq B_R(z) ; S_{\rho_\epsilon}(z) \ge\sin
(\theta
)\}\bigr] + c\theta^{2} \\
&&\qquad \le c\theta^{-\beta} \Prob^*\{\gamma[{\rho_\eta},{\rho
_\epsilon}]
\not\subseteq B_R(z)\} + c\theta^{2} \\
&&\qquad \le c\theta^{-\beta} \Prob^*\{\gamma[{\rho_\eta},T_z] \not
\subseteq
B_R(z)\} + c\theta^{2} \\
&&\qquad \le c\theta^{-\beta} (\eta/R)^{\alpha} + c\theta^{2},
\end{eqnarray*}
where $c$ is being used generically.
Thus by an appropriate choice of $\theta$, for example,
\[
\theta= (\eta/R)^{{\alpha}/({2+\beta})},
\]
we get the desired bound.
\end{pf}

\section{Multi-point Green's function}\label{multipoint}

In this section we consider two distinct points $z,w\in\Half$.
To simplify notation, we write
\begin{eqnarray*}
\xi &=& \xi_\epsilon= \xi_{z,\epsilon} = \inf\{t\dvtx \Upsilon_t(z)
\leq
\epsilon\},
\\
\chi &=& \chi_\delta= \chi_{w,\delta} = \inf\{t\dvtx \Upsilon_t(w) \leq
\delta\}.
\end{eqnarray*}
Although we will write $\xi,\chi$, it is important to remember
that these quantities depend on $z,\epsilon,w,\delta$. We let
$\Prob,\E$ denote probabilities and expectations for $\mbox{SLE}_\kappa$
from $0$ to $\infty$ in $\Half$ and $\Prob^*,\E^*$ for the
corresponding quantities for a two-sided radial through $z$.
The multi-point Green's function, which we write
\[
G(z,w) = G_\Half(z,w;0,\infty),
\]
roughly corresponds
to the probability that $\mbox{SLE}$ in $\Half$ from $0$ to $\infty$ goes
through $z$ and then through $w$. This quantity is not symmetric.
Although we do not have a closed from for this quantity, we can define
it precisely.\vadjust{\goodbreak}
\begin{definition*} The multi-point Green's function
$G(z,w)$ is defined by
\[
G(z,w) = G(z) \E^*[G_H(w;z,\infty)],
\]
where $H$ is the unbounded component of $\Half\setminus
\gamma(0,T_z]$.
\end{definition*}

It is worth noting that if $w$ is swallowed by the two-sided radial
$\mbox{SLE}$ curve before reaching $z$, this Green's function gives that event
weight zero since the curve $w$ is unreachable no matter how close the
curve was to $w$ before reaching $z$.

In~\cite{RS}, the exact formula for $G_\Half(z;0,\infty)$ was found by
considering the martingale $G_{H_t}(z,\gamma(t),\infty)$ and then using
It\^o's formula and scaling to find the ODE that it satisfies, which
could then be explicitly solved. When attempting the same technique
here, a three real variable PDE result, which does not immediately seem
to admit a closed form solution. A derivation of this PDE may be found
in Appendix~\ref{PDEcomp}.

The justification for this definition comes from the following
theorem. Implicit in the statement is that the limit can
be taken along any sequence of $\epsilon,\delta$ going to
zero.
\begin{theorem}\label{multitheorem} If $z,w \in\Half$, then
\[
\lim_{\epsilon,\delta\rightarrow0^+}
\epsilon^{d-2}
\delta^{d-2} \Prob\{\xi< \chi< \infty\}
= c_*^2 G(z,w) ,
\]
where $c_*$ is as defined in (\ref{mar41}).
\end{theorem}

When
\[
d = \biggl(1+\frac{\kappa}{8}\biggr) \wedge2
\]
is the dimension rather than simply $d = 1+\kappa/8$, this theorem
still defines an interesting quantity for $\kappa\ge8$. Since the
curve is space filling for $\kappa\ge8$, the limit is trivial and
\[
\lim_{\epsilon,\delta\rightarrow0^+}
\epsilon^{d-2}
\delta^{d-2} \Prob\{\xi< \chi< \infty\} = \Prob\{\xi_0 < \chi_0
\} =
c_*G(z,w).
\]
This agrees with the above definition of $G(z,w)$ since we may take
two-sided radial through $z$ for $\kappa\ge8$ to be the measure on
$\gamma$ stopped at the time the curve passes through $z$ and
\[
G_D(w;z_1,z_2) = 1\{w \in D\}.
\]
Since this case requires no further work, we will continue to assume
that $\kappa< 8$.

We will need one lemma that will follow from our work
on Beffara's estimate, which we will prove in Section~\ref{befSec}.
\begin{lemma}\label{mainBef} There exists $\alpha> 0$,
such that if
$z,w \in\Half$, then there exists $c = c_{z,w}< \infty,
$ such that for all $\epsilon,\delta,r >0$,
\[
\Prob\{\xi< \chi< \infty; \inrad_\xi(w) \le r\} \le c \epsilon
^{2-d}\delta^{2-d}r^\alpha.
\]
\end{lemma}

More precise results than this are obtained in this paper, but this is
all that is required in this section.

Before going through the details of the proof, we briefly sketch the
argument. To estimate
\[
\Prob\{\xi< \chi< \infty\},
\]
we wish to show that this probability is carried mostly
on curves which get within $\epsilon$ of $z$ in conformal radius before
decreasing the conformal radius of $w$ much at all. To show that the
curves which do not do this are negligible, we use Lem\-ma~\ref{mainBef}.

On the event that the curve stays bounded away from $w$, we know the
Green's function for getting to $w$ stays uniformly bounded, allowing us
to use convergence of the conditioned measures $\E[\cdot\,| \xi<
\infty
]$ to $\E^*[\cdot]$, the two-sided radial measure, as measures on the
$\mbox{SLE}$ curve up until some fixed conformal radius $\eta\gg\epsilon$.

This would be everything if it were not for the fact that the tip of
the curves (the portion very near $z$) under the conditioned measure
versus the two-sided radial measure have very different distribution.
To handle this, we use Lemmas~\ref{approachlem} and~\ref
{approachlemcond} to show that under both measures the tip stays close
to $z$ most of the time in Euclidean distance, and then Lemma \ref
{GreensLemma} tells us that the Green's function for getting to $w$ is
insensitive to these changes.

%
\begin{figure}

\includegraphics{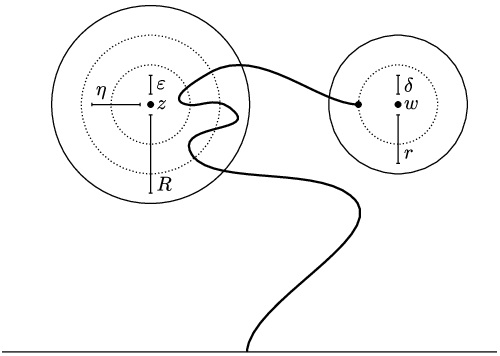}

\caption{A diagram of the proof of Theorem \protect\ref{multitheorem}.
Dotted circles represent conformal radii and solid circles refer to
geometric radii. The bold curve gives an example of the approximate
shape of a curve contributing to the leading order event.}
\label{GreensFig}
\end{figure}

To aid in the understanding of the proof, Figure~\ref{GreensFig} shows
diagrammatically the various distances considered and the approximate
shape of a curve in the main term.
\begin{pf*}{Proof of Theorem~\ref{multitheorem} given Lemma
\ref{mainBef}}
We first split according to how close we get to $w$ before getting
close to $z$. Fixing some $r < |z-w|/2$, by Lemma~\ref{mainBef} we see
that for some $\alpha> 0$
\begin{eqnarray*}
\Prob\{\xi< \chi< \infty\} & = &\Prob\{\xi< \chi< \infty; \inrad
_\xi
(w) > r\} \\
&&{} + \Prob\{\xi< \chi< \infty; \inrad_\xi(w) < r\} \\
& = &\Prob\{\xi< \chi< \infty; \inrad_\xi(w) > r\} + O(\epsilon
^{2-d}\delta^{2-d}r^\alpha).
\end{eqnarray*}

Let $\F_\xi$ denote the $\sigma$-algebra generated by the stopping
time $\xi$. By applying Lemma~\ref{fastGreen} to $w$ in the domain
$H_\xi$, we see if $\delta\leq r/2$,
\begin{eqnarray*}
&&\Prob\{\xi< \chi< \infty; \inrad_\xi(w) > r \mid\F_\xi\} \\
&&\qquad = 1\{\xi< \infty; \inrad_\xi(w) > r\} c_* \delta^{2-d}
G_{H_\xi}(w;\gamma(\xi),\infty) \bigl[1+O\bigl((\delta/r)^u\bigr)\bigr] .
\end{eqnarray*}

Applying Lemma~\ref{fastGreen} to $z$ in $\Half$ combined with the
previous equation implies
\begin{eqnarray*}
&&c_*^{-2} \epsilon^{d-2}\delta^{d-2} G_\Half(z;0,\infty)^{-1}
\Prob\{\xi< \chi< \infty; \inrad_\xi(w) > r\} \\
&&\qquad = \bigl[1+O\bigl(\epsilon^u + (\delta/r)^u\bigr)\bigr]
\E[G_{H_\xi}(w;\gamma(\xi), \infty)1\{\inrad_\xi(w)>r\} | \xi
<\infty].
\end{eqnarray*}
For simplicity of notation, given a stopping time $\tau$, we let
\[
\E_\tau[ \cdot] = \E[ \cdot\,| \tau<\infty] \quad\mbox{and}\quad G_{\tau
,r} =
G_{H_\tau}(w;\gamma(\tau), \infty)1\{\inrad_\tau(w) > r\},
\]
and hence we may rewrite this as
\begin{eqnarray*}
&&\Prob\{\xi< \chi< \infty; \inrad_\xi(w) > r\} \\
&&\qquad = c_*^2 \epsilon^{2-d}\delta^{2-d}G_\Half(z;0,\infty
)\bigl[1+O\bigl(\epsilon
^u+(\delta/r)^u\bigr)\bigr]\E_\xi[G_{\xi,r} ].
\end{eqnarray*}

We wish to transform this expression from the conditioned measure to
the two-sided radial measure, and from considering the situation at
time $\xi$ (the time it first gets within conformal radius $\epsilon$)
to $T_z$ (the time under the two-sided radial measure that $z$ is first
contained in the boundary of $H_{T_z}$). To do so we will pass through
a series of steps.

Fix some $\eta, R$ so that $\epsilon< \eta< R < |z-w|/2$. We wish to
control the difference
\begin{eqnarray*}
|\E_\xi[G_{\xi,r}] - \E_\xi[G_{\xi_\eta,r}]| & \le & \E_\xi
\bigl[|G_{\xi,r} -
G_{\xi_\eta,r}| 1\{\gamma[\xi_\eta,\xi] \subseteq B_R(z)\}\bigr] \\
&&{} + \E_\xi\bigl[|G_{\xi,r} - G_{\xi_\eta,r}| 1\{\gamma[\xi_\eta,\xi
] \not
\subseteq B_R(z)\}\bigr].
\end{eqnarray*}
By Lemma~\ref{GreensLemma} and the fact that the inradius about $w$
cannot decrease between $\xi_\eta$ and $\xi$ if $\gamma[\xi_\eta
,\xi]
\subseteq B_R(z)$, we see that
\[
\E_\xi\bigl[|G_{\xi,r} - G_{\xi_\eta,r}| 1\{\gamma[\xi_\eta,\xi]
\subseteq
B_R(z)\}\bigr] = O\bigl(r^{d-2-(\beta\wedge1)/2}R^{(\beta\wedge1)/2}\bigr).
\]
On the second term, the difference is no bigger than $O(r^{d-2})$ on an
event, which by Lemma~\ref{approachlemcond} is $O((\eta/R)^{\alpha'})$
for some $\alpha' > 0$. Putting it all together yields
\[
|\E_\xi[G_{\xi,r}] - \E_\xi[G_{\xi_\eta,r}]| = O\bigl(r^{d-2-(\beta
\wedge
1)/2}R^{(\beta\wedge1)/2} + r^{d-2}(\eta/R)^{\alpha'}\bigr).
\]

By Lemma~\ref{measurelimit}, we know for events in $\F_{\xi_\eta}$
we have
\[
\biggl|\frac{\mathrm{d} \Prob^*}{\mathrm{d} \Prob_\xi} - 1\biggr| =
O\bigl((\epsilon/\eta)^u\bigr),
\]
and hence we have
\[
|\E_\xi[G_{\xi_\eta,r}] - \E^*[G_{\xi_\eta,r}]| =
O\bigl(r^{d-2}(\epsilon
/\eta)^u\bigr).
\]

Analogously to before, consider splitting the difference
\begin{eqnarray*}
|\E^*[G_{\xi_\eta,r}] - \E^*[G_{T_z,r}]| & \le & \E^*\bigl[|G_{\xi_\eta
,r} -
G_{T_z,r}| 1\{\gamma[\xi_\eta,T_z] \subseteq B_R(z)\}\bigr] \\
&&{} + \E^*\bigl[|G_{\xi_\eta,r} - G_{T_z,r}| 1\{\gamma[\xi_\eta,T_z]
\not
\subseteq B_R(z)\}\bigr].
\end{eqnarray*}
By Lemma~\ref{GreensLemma} and the fact that the inradius about $w$
cannot decrease between $\xi_\eta$ and $T_z$ if $\gamma[\xi_\eta,T_z]
\subseteq B_R(z)$, we again see
\[
\E^*\bigl[|G_{\xi_\eta,r} - G_{T_z,r}| 1\{\gamma[\xi_\eta,T_z]
\subseteq
B_R(z)\}\bigr] = O\bigl(r^{d-2-(\beta\wedge1)/2}R^{(\beta\wedge1)/2}\bigr).
\]
The second term is on an event which by Lemma~\ref{approachlem}
is $O((\eta/R)^{\alpha'})$, and hence
\[
|\E^*[G_{\xi_\eta,r}] - \E^*[G_{T_z,r}]| = O\bigl(r^{d-2-(\beta\wedge
1)/2}R^{(\beta\wedge1)/2} + r^{d-2}(\eta/R)^{\alpha'}\bigr).
\]

We may easily see that
\[
\Prob^*\{\inrad_{T_z}(w) = 0\} \le\sum_{k \ge1} \Prob^*\{\inrad
_{\xi
_{1/k}}(w) = 0\} = 0
\]
by the fact that $\Prob^*$ is absolutely continuous with respect to
$\Prob$ until the stopping time $\xi_{1/k}$ combined with that fact
that two-sided radial $\mbox{SLE}$ generates a curve with probability one.
Hence, since $G_{T_z}(w;z,\infty) \ge0$, we have that
\[
\E^*[G_{T_z}(w;z,\infty) 1\{\inrad_{T_z}(w) > r\}] \rightarrow\E
^*[G_{T_z}(w;z,\infty)] \qquad\mbox{as } r \rightarrow0.
\]

Combining all these terms and by combining exponents, we see there
exists some $b > 0$ such that
\begin{eqnarray*}
&&
\epsilon^{d-2}\delta^{d-2}\Prob\{\xi< \chi< \infty\}\\
&&\qquad =
c_*^2G_\Half
(z;0,\infty)\bigl[1+O\bigl(\epsilon^b+(\delta/r)^b\bigr)\bigr]
\E^*[G_{T_z,r}] \\
&&\qquad\quad{} +O\bigl(r^b + (R/r)^b+(R/r)^b(\eta/R)^b
+(\epsilon/r)^b(\epsilon/\eta)^b\bigr).
\end{eqnarray*}
Thus by choosing $r$, $\eta$ and $R$ so that as $\epsilon, \delta
\rightarrow0$ we also have
\begin{eqnarray*}
r &\rightarrow&0,\qquad \delta/r \rightarrow0,\qquad
\epsilon/r \rightarrow0,\\
R/r&\rightarrow&0,\qquad \eta/R \rightarrow0,\qquad
\epsilon/\eta\rightarrow0,
\end{eqnarray*}
we see that
\[
\epsilon^{d-2}\delta^{d-2}\Prob\{\xi< \chi< \infty\} \rightarrow
c_*^2G_\Half(z;0,\infty)\E^*[G_{T_z}(w;z,\infty)]
\]
as desired.
\end{pf*}

This same argument generalizes to show that we can define higher-order
Green's functions of $\mbox{SLE}$ as well (those that give normalized
probabilities for passing through $k$ marked points in the interior),
and that the resulting multi-point Green's functions can be written in
terms of expectations under the two-sided radial measure of lower-order
Green's functions, for instance,
\begin{eqnarray*}
&&\lim_{\epsilon_1,\epsilon_2,\epsilon_3\rightarrow0}\epsilon
_1^{d-2}\epsilon_2^{d-2}\epsilon_3^{d-2}\Prob\{\xi_{\epsilon
_1,z_1}<\xi
_{\epsilon_2,z_2}<\xi_{\epsilon_3,z_3}\}\\
&&\qquad = c_*^3G_\Half(z_1;0,\infty)\E
^*[G_{H_{T_{z_1}}}(z_2,z_3;z_1,\infty)],
\end{eqnarray*}
where $\E^*$ is the two-sided radial measure passing through $z_1$.

Note that we may obtain the multi-point Green's function as defined in
the \hyperref[intro]{Introduction} by summing this over the case where it gets near to
$z$ then $w$ and the case where it gets near to $w$ then $z$.

The remainder of this paper is dedicated to providing a proof of Lemma~\ref{mainBef} and a sharpened version of Beffara's estimate.

\section{\texorpdfstring{Proof of Beffara's estimate and Lemma \protect\ref{mainBef}}
{Proof of Beffara's estimate and Lemma 3.1}}\label{befSec}

To complete our proof of the existence of multi-point Green's functions
we require a proof of Lemma~\ref{mainBef}. We also wish to prove
Befarra's estimate which is the following theorem.
\begin{theorem}[(Beffara's estimate)] \label{befsEst}
There exists a $c > 0$ such that for all $z$, \mbox{$w \in\Half$} with $\Im
(z),\Im(w) \ge1$ we have that
%
\begin{equation} \label{octbef}
\Prob\{\Upsilon_\infty(z)<\epsilon, \Upsilon_\infty(w)<\delta\}
\le
c\epsilon^{2-d}\delta^{2-d}|z-w|^{d-2}.
\end{equation}
\end{theorem}

The hard work will be establishing the result when $z,w$ are far apart.
We use the notation introduced in Section~\ref{notations}. For later
convenience, we write
this proposition in terms of the usual radius rather than the
conformal radius, but it is easy
to convert to conformal radius using the Koebe $(1/4)$-theorem. We use
the notation
\[
\Delta_t(z) = \inrad_{H_t}(z).
\]

\begin{proposition} \label{main2}
For every $0 < \theta< \infty$, there exists
$c < \infty$, such that if $z,w \in\Half$ with
\[
\Im(z),\Im(w) \ge\theta\quad\mbox{and}\quad |z-w| \geq\theta/9 ,\vadjust{\goodbreak}
\]
then
\[
\Prob\{\Delta_\infty(z) \leq\epsilon,
\Delta_\infty(w) \leq\delta\}
\leq c \epsilon^{2-d} \delta^{2-d} .
\]
\end{proposition}
\begin{pf*}{Proof of Theorem~\ref{befsEst} given Proposition
\ref{main2}}
Without loss of generality we assume that $1 \leq\Im(z) \leq\Im(w)$.
We first
claim that is suffices to prove (\ref{octbef}) when $1 = \Im(z) \leq
\Im
(w)$. Indeed, if
this is true and $r > 1$, scaling implies that
\begin{eqnarray*}
\Prob\{\Upsilon_\infty(rz)<\epsilon, \Upsilon_\infty(rw)<\delta\}
& = & \Prob\{\Upsilon_\infty(z) < \epsilon/r, \Upsilon_\infty(w) <
\delta
/r\}\\
& \leq & c (\epsilon/r)^{2-d} (\delta/r)^{2-d} |z-w|^{d-2} \\
& < & c \epsilon^{2-d} \delta^{2-d}
|rz - rw|^{d-2}.
\end{eqnarray*}

Suppose $\epsilon>
|z-w|/10$. Then, using the one-point estimate Lemma~\ref{fastGreen},
we get
\begin{eqnarray*}
\Prob\{\Upsilon_\infty(z)<\epsilon, \Upsilon_\infty(w)<\delta\}&
\leq&\Prob\{\Upsilon_\infty(w) < \delta\}\\
& \leq &c \delta^{2-d} \\
& \leq &c 10^{2-d}
\epsilon^{2-d} \delta^{2-d} |w-z|^{d-2}.
\end{eqnarray*}
A similar argument with $\delta$ shows that it suffices to prove
(\ref{octbef}) with $\Im(z) = 1$ and $\epsilon,\delta< |z-w|/10$.
If $|z-w| \geq1/9$, we can apply Proposition~\ref{main2} directly.
So for the remainder of the proof, we let $u = |z-w|$ and assume
\[
1 = \Im(z) \leq\Im(w),\qquad u \leq\frac19,\qquad
\epsilon,\delta\leq\frac u{10}.
\]

We will use the growth and distortion theorems which we now
recall (see, e.g.,~\cite{Lbook}, Section 3.2).
Suppose $f\dvtx\Disk\rightarrow\C$ is a univalent function
with $f(0) = 0, |f'(0)| = 1$. Then if $|\zeta| < 1$,
%
\begin{eqnarray} \label{growth}
\frac{|\zeta|}{(1 + |\zeta|)^2} &\leq&|f(\zeta)|
\leq\frac{|\zeta|}{(1-|\zeta|)^2} ,
\\
%
\label{distortion}
\frac{1-|\zeta|}{(1+|\zeta|)^3} &\leq&|f'(\zeta)| \leq\frac{1
+|\zeta|}
{(1 - |\zeta|)^3}.
\end{eqnarray}

Let $\tau= \inf\{t\dvtx |\gamma(t) - z| = 8u\}
= \inf\{t\dvtx \Delta_t(z) = 8u\}$. The triangle inequality implies
that $7u \leq\Delta_\tau(w) \leq9u$.
Lemma~\ref{fastGreen}
implies that
%
\begin{equation} \label{apr71}
\Prob\{\tau< \infty\} \leq c u^{2-d}.
\end{equation}
Let $g_\tau$ be the usual conformal map, and let $h = s g_\tau$ where
$s > 0$ is chosen so that $\Im(h(z)) = 1$. By the Schwarz lemma and the
Koebe $(1/4)$-theorem, $4u \leq\Upsilon_{H_{\tau}}(z) \leq16u$,
and since $\Upsilon_\Half(h(z)) = 1$,
\[
\frac{1}{16u} \leq|h'(z)| \leq\frac{1}{4u} .
\]
Since $h$ is a conformal transformation of the disk of radius $8u$
about $z$, (\ref{growth}) implies
\[
\tfrac{4}{81}
\leq(8/9)^2 u |h'(z)| \leq
|h(w) - h(z)| \leq(8/7)^2 u |h'(z)| \leq\tfrac{16}{49}.
\]
Since $\epsilon\leq u/10 = (8u)/80$, if $|z-\zeta| \leq\epsilon$,
(\ref{growth})
implies
\[
|h(\zeta) - h(z)| \leq\biggl( \frac{80}{79}\biggr)^2 \epsilon
|h'(z)| \leq\biggl( \frac{80}{79}\biggr)^2 \frac{\epsilon}{4u}
\leq\frac{2\epsilon}
{7u}.
\]
Applying (\ref{distortion}), we can see that
\[
|h'(w)| \leq\frac{(10/9)}{(8/9)^3} |h'(z)| \leq\frac{(10/9)}{(8/9)^3}
\frac{1}{4u} .
\]
Applying (\ref{growth}) to the disk of radius $7u$ about $w$ and using
$
\delta\leq u/10 = (7u)/70$, we see that for $|\zeta- w| \leq\delta$,
\[
|h(\zeta) - h(w) | \leq\biggl(\frac{70}{69}\biggr)^2 \delta|h'(w)|
\leq\biggl(\frac{70}{69}\biggr)^2 \delta\frac{(10/9)}{(8/9)^3}
\frac{1}{4u} \leq\frac{3 \delta}{7u}.
\]

Using the estimates of the previous paragraph, we can see by
conformal invariance and the Markov property, that
\[
\Prob\{\Delta_\infty(z) \leq\epsilon, \Delta_\infty(w) \leq
\delta\mid\tau< \infty\}
\]
is bounded above by the supremum of
\[
\Prob\{\Delta_\infty(z') \leq\epsilon', \Delta_\infty(w') \leq
\delta' \},
\]
where the supremum is over
\[
\Im(z') = 1,\qquad \frac4{81} \leq|z'-w'| \leq\frac{16}{49} ,\qquad
\epsilon' \leq\frac{2 \epsilon}{7u} ,\qquad \delta' \leq\frac{3 \delta
}{7u}.
\]
Proposition~\ref{main2} implies that there exists $c'$ such that this supremum
is bounded by
\[
c' (\epsilon')^{2-d} (\delta')^{2-d} \leq
c \epsilon^{2-d} \delta^{2-d}
u^{2(d-2)}.
\]
If we combine this with (\ref{apr71}), we get
\[
\Prob\{\Delta_\infty(z') \leq\epsilon', \Delta_\infty(w') \leq
\delta' \} \leq c \epsilon^{2-d} \delta^{2-d} u^{d-2},
\]
which is what we needed to prove.
\end{pf*}

By an analogous argument to how we obtained Theorem~\ref{befsEst} from
Proposition~\ref{main2}, we may obtain Lemma~\ref{mainBef} from
Proposition~\ref{fromBef}.
\begin{proposition}\label{fromBef}
For every $0 < \theta< \infty$, there exists $c < \infty$ and $\alpha
> 0$ such that if $z,w \in\Half$ with
\[
\Im(z),\Im(w) \ge\theta\quad\mbox{and}\quad |z-w| \ge\theta/9,\vadjust{\goodbreak}
\]
then for $\rho> \delta$
%
\begin{equation}\label{main34}
\Prob\{\Delta_\infty(z) \le\epsilon, \Delta_\infty(w) \le\delta,
\Delta_\sigma(w) \le\rho\} \le c \epsilon^{2-d}\delta^{2-d}\rho
^\alpha,
\end{equation}
where $\sigma= \inf\{ t \dvtx \Delta_t(z) \le\epsilon\mbox{ or } \Delta_t(w)
\le\delta\}$.
\end{proposition}

This proposition will follow immediately from the work required to show
Proposition~\ref{main2}.

To prove the proposition, we will show that there exists a $c < \infty$
such that
(\ref{main34}) holds if $|z-w| \ge2\sqrt{2}$ and $\Im(z),(w) \ge1$.
By scaling one can easily deduce the result for all $\theta>0$ with
a $\theta$-dependent constant.
We fix $z,w$ with $|z-w| \ge2\sqrt{2}$ and $\Im(z),\Im(w) \ge1$, and
denote by $\imax$ some fixed vertical or diagonal line such that
%
\begin{equation}\label{iCond}
\dist(z,\imax),\dist(w,\imax) \ge1 ,
\end{equation}
and $z,w$ lie in different components of $\Half\setminus\imax$. We will
further assume, without loss of generality, that $z$ is in the
component of $\Half\setminus\imax$ which contains arbitrarily large
negative real numbers in it's boundary (more informally that $z$ is in
the left component).

\subsection{An excursion measure estimate}

Our main result will require an estimate of the ``distance'' between
two boundary arcs in a simply connected domain. We will use excursion
measure to gauge the distance; we could also use extremal distance, but
we find excursion measure more convenient.

Suppose $\eta$ is a crosscut in $\Half$ with $-\infty< \eta(1^-)
\leq\eta(0^+) \le-1$. Let
$H_\eta$ denote the unbounded component of $\Half\setminus
\eta$.
Let $\exc(\eta) = \exc_{\Half_\eta}(\eta,[0,\infty))$
denote the excursion measure between $\eta$ and $[0,\infty)$ in
$H_\eta$, the definition of which we now
recall (see~\cite{Lbook}, Section 5.2, for more details).
If $z \in H_\eta$, let $h_\eta(z)$ be the probability that a Brownian
motion starting at $z$ exits $H_\eta$ at $\eta$. For $x \geq0$,
let $\partial_y h_\eta(x)$ denote the partial derivative. Then
\[
\exc(\eta) = \int_0^\infty\partial_y h_\eta(x) \,dx.
\]

The excursion measure $\exc_D(V_1,V_2)$ is defined for any domain
and boundary arcs $V_1,V_2$ in a similar way and is a conformal
invariant. If $V_2$ is smooth, then we can compute
$\exc_D(V_1,V_2)$ by a similar integral
\[
\exc_D(V_1,V_2) = \int_{V_2} \partial_{\mathbf{n}} h_{V_1}(z)
|dz| ,
\]
where $\mathbf{n}$ denotes the inward normal. We need the following
easy estimate.
\begin{lemma} \label{lemmaoct21}
There exist $c_1,c_2$ such that if $\eta$ is a crosscut in $\Half$
with $-\infty<
\eta(1^-) \leq\eta(0^+)=-1$ and $\diam(\eta) \leq1/2$, then
\[
c_1 \diam(\eta) \leq\exc(\eta) \leq c_2 \diam(\eta).
\]
\end{lemma}
\begin{pf*}{Sketch of proof} In fact, we get an estimate
\[
\partial_{y} h_{\eta}(x)
\asymp\frac{\diam(\eta)}{(x+1)^2}.
\]
The key estimate used here is the fact that
that if $\Re(z) \geq0$,
\[
h_\eta(z) \asymp\frac{\Im(z) \diam(\eta)}
{(|z|+1)^2} .
\]
\upqed\end{pf*}
\begin{lemma}\label{jul16lemma1}
There exists a $C < \infty$ such that the following is true.
Suppose $H \subset\C$ is a half-plane bounded by the line $L =
\partial H$,
$D$ is a simply connected subdomain of $\Half$ and $z \in\partial D$
with
$d(z,L) > \frac{1}{2}$.
Suppose $I$ is a subinterval of $ L \cap\partial D$. Then for every
$\epsilon< \frac{1}{2}$, the
excursion measure between $I$ and $V:= \partial D \cap\{w\dvtx |w-z| \leq
\epsilon\}$ is bounded above by $C
\epsilon^{1/2}$.
\end{lemma}
\begin{pf} Without loss of generality we assume that $H = \Half,
z = i/2$. Let $h(w)$ denote the probability that a Brownian
motion starting at $w$ exits $D$ at $V$. Then the excursion measure is
exactly
\[
\int_I \partial_y h(x) \,dx .
\]
Hence it suffices to give an estimate
%
\begin{equation} \label{jul172}
\partial_y h(x) \leq c \epsilon^{1/2}
[1 \wedge x^{-2}].
\end{equation}
For $|x| \leq4$, this follows from the Beurling estimate.
For $|x| \geq4$, we first
consider the excursion ``probability'' to reach $\Re(w) = x/2$. By
the gambler's ruin estimate, this is bounded by
$O(|x|^{-1})$. Then we need to consider the probability that
a Brownian motion starting at $z'$ with $\Re(z') = x/2$ reaches
the disk of radius $1$
about $z$ without leaving $D$. By comparison with the same probability
in the domain $\Half$, we see that this is bounded above by $O(|x|^{-1})$.
Finally from
there we need to hit $V$ which contributes a factor of $O(\epsilon
^{1/2})$ by
the Beurling estimate. Combining these estimates gives (\ref{jul172}).
\end{pf}
\begin{lemma}\label{FullExcursion}
There exists $c > 0$ such that the following holds.
Let $D$ be a simply connected domain, and let $\gamma$ be a
chordal $\mbox{SLE}_\kappa$
path from $z_1$ to $z_2$ in $D$. Let $\eta\dvtx(0,1) \rightarrow D$ be a
crosscut in
$D$. Let $\xi\dvtx (0,1) \rightarrow D$ be another
crosscut with $\xi(0^+) = z_1$, and let $D_1,D_2$ denote the
components of $D \setminus\xi$. Suppose $\eta\subset D_1$ and
$z_2 \in\partial D_2$.
Then,
\[
\Prob\{\gamma(0,\infty) \cap\eta(0,1) \neq\varnothing\} \le c
\exc
_{D}(\eta,\xi)^
{\beta}.
\]
\end{lemma}

See Figure~\ref{excFig} for a diagram of the setup of this lemma.
\begin{pf*}{Proof of Lemma~\ref{FullExcursion}} By conformal invariance, we may assume that
$D = \Half$,
$z_1 = 0, z_2 = \infty$, and it suffices to prove the result when
$\exc_{D}(\eta,\xi) \le1$ in which case
%
\begin{figure}

\includegraphics{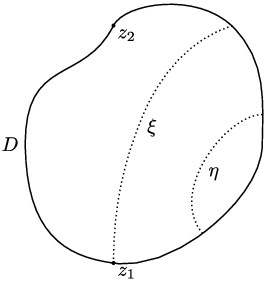}

\caption{The setup for Lemma \protect\ref{FullExcursion}.}
\label{excFig}
\end{figure}
the endpoints of $\eta$ are nonzero. Without loss of
generality we assume that they lie on the negative real axis, and by
scale\vadjust{\goodbreak} invariance we may assume
$\eta(1^-) \leq\eta(0^+) = -1$. Then
monotonicity of the excursion measure implies that
\[
\exc_{D}(\eta,\xi) \geq\exc_D(\eta).
\]
Lemma~\ref{lemmaoct21} implies that if $\diam(\eta) < 1/2$, then
$\exc_D(\eta) \asymp\diam(\eta)$. Since \mbox{$\exc_D(\eta) \le1$} one can
see there is a $c_0$ so that $\diam(\eta) \le c_0$.
The result then follows from Proposition~\ref{julprop}.
\end{pf*}

Given the proof, the form of this lemma may seem odd as the curve $\xi$
is discarded half way through; indeed, the result could be stated with
$\exc_D(\eta)$ rather than $\exc_D(\eta,\xi)$ in the inequality.
However, $\exc_D(\eta)$ is hard to estimate directly and, in every case
in this paper, the method of estimation is to find a curve $\xi$ and
proceed as above.

\subsection{Topological lemmas}

The most challenging portion of this proof is gaining simultaneous
control of
the distances to the near and far edges of the curve. Luckily, we may
eliminate a number of hard cases of the computations that follow by
purely topological means. For clarity of presentation, we have isolated
these topological lemmas here in a separate section.
Let $z,w,\imax$ be as described in the paragraph around equation (\ref
{iCond}). We call $\gamma$ a noncrossing curve (from $0$
to $\infty$ in~$\Half$) if is generated by the Loewner equation
(\ref{chordaleq}) with some
driving function $U_t$, and, as before, we let $H_t$ be the unbounded
component of
$\Half\setminus\gamma(0,t]$ and $\partial_1 H_t, \partial_2 H_t$ be
the preimages (considered
as prime ends) under
$g_t$ of $(-\infty,U_t)$ and $(U_t, \infty)$. We call a simple curve
$\omega\dvtx(0,\infty)
\rightarrow H_t$ with $\omega(0^+) = \gamma(t)$ and $\omega(\infty) =
\infty$ an
\textit{infinite crosscut of $H_t$}. Such curves can be obtained as
preimages under
$g_t$ of simple curves from $U_t$ to $\infty$ in $\Half$.
An important observation is that infinite crosscuts of $H_t$
separate $\partial_1 H_t$ from $\partial_2 H_t$.

We now define a particular crosscut of $H_t$ contained in $\imax$ that
separates $z$ from~$w$.
\begin{definition*}
Let $\gamma$ be a noncrossing curve, and let
$\imax_t = \imax\setminus\gamma(0,t]$. We denote by $I_t =
I_t(\imax
,z,w,\gamma)$
the unique open interval contained in $\imax$ such that the following
four properties hold. For any $t \le t'$ we have:
\begin{itemize}
\item$I_t$ is a connected component of $\imax_t$,
\item$I_{t'} \subseteq I_t$,
\item$H_t \setminus I_t$ has exactly two connected components, one
containing $z$ and one containing $w$ and
\item$I_t = I_{t'}$ whenever $\gamma(t,t'] \cap\imax= \varnothing$.
\end{itemize}
We let $H_t^z,H_t^w$ denote the components of $H_t \setminus I_t$ that
contain $z$ and $w$, respectively.
\end{definition*}

Seeing that this notion is well defined is nontrivial, despite the
intuitive nature what it should be (see Figure~\ref{itfig}). To avoid
%
\begin{figure}

\includegraphics{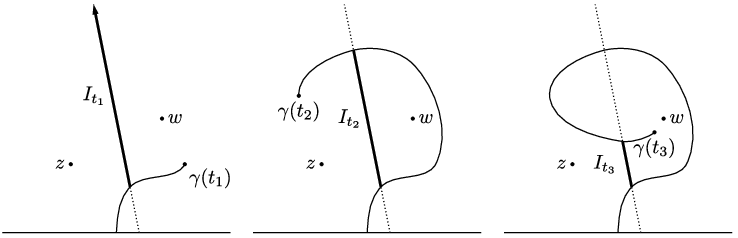}

\caption{A few steps showing the behavior of $I_t$ for some times $0 <
t_1 < t_2 < t_3$.}
\label{itfig}
\end{figure}
breaking the flow of the document, the proof that it is well defined
has been deferred to Appendix~\ref{ItAppendix}.

\begin{lemma}\label{boundedLemma}
Suppose $\gamma$
is a noncrossing curve with $z,w \notin\gamma(0,\infty) $ and $I_t
= I_t(\imax,z,w,\gamma)$
as above. Suppose $\gamma(t) \in\overline{I}_{t}$. If $I_t$ is not
bounded, then
\[
\Delta^*_{H_t}(z,\gamma(t),\infty) \ge1 ,\qquad \Delta^*_{H_t}(w,\gamma
(t),\infty) \ge1.
\]
\end{lemma}
\begin{pf} Suppose $I_t$ is not bounded. Then $I_t$ is an infinite
crosscut of $H_t$.
Suppose that $\Delta^*_{H_t}(z,\gamma(t),\infty) < 1$. Then there is a
crosscut $\eta$ contained in a disc of radius strictly less than one
centered on $z$ which has one end point in $\partial_1H_t$ and one end
point in $\partial_2H_t$.
Hence $\eta$ must intersect $I_t$. However, $\dist(z,I_t) \geq\dist
(z,\imax) \geq1$ which is a contradiction.
Therefore, $\Delta^*_{H_t}(z,\gamma(t),\infty) \ge1$.
\end{pf}
\begin{lemma}\label{oneSidedInfinite}
Suppose $\gamma$
is a noncrossing curve with $z,w \notin\gamma(0,\infty) $ and $I_t
= I_t(\imax,z,w,\gamma)$
as above. Suppose $\gamma(t) \in\overline{I}_{t}$. If $I_t$ is
bounded, and $H_t^z$ is bounded, then
\[
\Delta^*_{H_t}(z,\gamma(t),\infty) \ge1.
\]
\end{lemma}
\begin{pf}
Suppose $I_t$ is bounded, $H_t^z$ is bounded, and
$\Delta^*_{H_t}(z,\gamma(t),\infty) < 1$. Then there is a crosscut
$\eta$ of $H_t^z$ which has one end point in $\partial_1H_t$ and one
end point in $\partial_2H_t$. Since $H_t^z$ is bounded and $\gamma(t)
\in\overline{I}_t$, we may find an infinite crosscut $\omega$
of $H_t$
that never enters $H_t^z$ [take a simple curve from $\infty$ in $H_t$
until it first hits $I_t$ and then continue
the curve along $I_t$ to reach $\gamma(t)$]. Since $\eta$ and $\omega$
do not intersect,
we get a contradiction.\vspace*{-2pt}
\end{pf}

Given these simple observations, we can restrict the manner in which
the various distances to the curve can be decreased.\vspace*{-2pt}
\begin{lemma}
Suppose $\gamma$
is a noncrossing curve with $z,w \notin\gamma(0,\infty) $ and $I_t
= I_t(\imax,z,w,\gamma)$
as above. Suppose $t_0$ is a time so that $\gamma(t_0) \in\overline
{I}_{t_0}$. Let $\zeta= \inf\{t > t_0
| \gamma(t) \in I_{t^-}\}$. Then at most one of the following holds:
\begin{itemize}
\item$\Delta_{H_{\zeta},1}(z,\gamma(\zeta),\infty) < \Delta
_{H_{t_0},1}(z,\gamma(t_0),\infty) \wedge1$, or
\item$\Delta_{H_{\zeta},2}(z,\gamma(\zeta),\infty) < \Delta
_{H_{t_0},2}(z,\gamma(t_0),\infty) \wedge1$.\vspace*{-2pt}
\end{itemize}
\end{lemma}
\begin{pf}
If $\zeta= t_0$, the above statement follows immediately, so we may
assume $\zeta> t_0$.
Consider the noncrossing loop $\ell= \gamma[t_0,\zeta] \cup L$ where
$L$ is the line connecting
$\gamma(\zeta)$ and $\gamma(t_0)$.
Partition $\Half$ into two sets, the infinite component of $\Half
\setminus\ell$, which we will denote by $A_\infty$, and the union of
the finite components of $\Half\setminus\ell$, which we will denote
by $A_0$. The point $z$ is either in $A_\infty$ or $A_0$. As the cases
are similar, assume $z \in A_\infty$. Since $\ell$ is a noncrossing
loop, we either have a curve $\eta\dvtx[0,1) \rightarrow A_\infty$ with
$\eta(0) = z$ and $\eta(1^-) \in\partial_1 H_\zeta$ or $\eta(1^-)
\in
\partial_2 H_\zeta$, but not both. This yields that only one of the
$\Delta_{H_{\zeta},j}(z,\gamma(\zeta),\infty)$ could have decreased
past the minimum of $1$ and its previous value.\vspace*{-2pt}
\end{pf}
\begin{lemma}\label{boundLemma}
Suppose $\gamma$
is a noncrossing curve with $z,w \notin\gamma(0,\infty) $ and $I_t
= I_t(\imax,z,w,\gamma)$
as above. Suppose $t_0$ is a time so that $\gamma(t_0) \in\overline
{I}_{t_0}$,
and let $\zeta= \inf\{t > t_0 | \gamma(t) \in I_{t^-}\}$. Suppose for
some $s < 1$,
\[
\Delta^*_{\zeta}(z) \le s < \Delta^*_{t_0}(z) .
\]
Then $\Delta_{t_0}(z) \le s$, and $H_{t_0}^w$ and $H_{\zeta}^w$
are bounded.\vspace*{-2pt}
\end{lemma}
\begin{pf}
By the previous lemma, we have that either $\Delta^1_{\zeta}(z) \ge
\Delta^1_{t_0}(z) \wedge1$ or $\Delta^2_{\zeta}(z) \ge\Delta
^2_{t_0}(z) \wedge1$. This implies that $\Delta^*_{\zeta}(z) \ge
\Delta
_{t_0}(z) \wedge1$,
and hence
$\Delta_{t_0}(z) \wedge1 \le s$ which is the first assertion.

We now prove that $H_\zeta^w$ is bounded. Assume first that both
$H_\zeta^w$ and $H_\zeta^z$ are unbounded. Then $I_\zeta$ is unbounded,
and by Lemma~\ref{boundedLemma} we have that
\[
\Delta^*_{H_\zeta}(z,\gamma(t),\infty) \ge1,
\]
which is a contradiction. Thus one of $H_\zeta^w$ or $H_\zeta^z$ is
bounded. If $H_\zeta^z$ is bounded, then by Lemma \ref
{oneSidedInfinite} we have
\[
\Delta^*_{H_\zeta}(z,\gamma(t),\infty) \ge1,
\]
which is again a contradiction. Thus $H_\zeta^w$ is bounded, as desired.\vadjust{\goodbreak}

By the definition of $\zeta$ and $I_t$, we know $\gamma(t_0,\zeta)$ is
contained in precisely one of $H_{t_0}^z$ or $H_{t_0}^w$. Since
\[
\Delta^*_{\zeta}(z) < 1 \le\Delta^*_{t_0}(z)
\]
by assumption, we know $\gamma(t_0,\zeta)\subseteq H_{t_0}^z$. Assume
that $H_{t_0}^w$ were unbounded. Then there is a curve $\eta$ from $w$
to $\infty$ contained in $H_{t_0}^w$. Since $H_\zeta^w$ is bounded
$\eta\cap\bd H_\zeta^w$ is nonempty. By definition,
\[
\bd H_\zeta^w \subseteq\gamma(0,t_0] \cup\gamma(t_0,\zeta] \cup
I_\zeta.
\]
We now show $\eta$ cannot intersect any of the three sets on the right.
Since $\eta$ is in $H_{t_0}^w$, we know $\eta\cap(\gamma(0,t_0]
\cup
I_{t_0}) = \varnothing$ and moreover, since $I_\zeta\subseteq
I_{t_0}$, that $\eta\cap I_\zeta= \varnothing$. Since $\gamma
(t_0,\zeta)\subseteq H_{t_0}^z$, we know $\eta\cap\gamma(t_0,\zeta
) =
\varnothing$. Thus we have a contradiction, and $H_{t_0}^w$ must be
bounded, as desired.
\end{pf}

\subsection{Main $\mbox{SLE}$ estimates}

We now use the above topological restrictions to help us establish the
needed $\mbox{SLE}$ estimates. Let $T_z$ (resp., $T_w$) denote the first time
that $z$ (resp., $w$) is not in $H_t$, and let $T = T_z \wedge T_w$
denote the first time that one of $z,w$ is not in $H_t$. Note that if
the curve is to approach $z$ and $w$ to within $\epsilon$ and $\delta$
as desired, it must do so before $T_z \vee T_w$.

We also define the following recursive set of stopping times. Let $\tau
_0 = 0$. Given $\tau_j< T$, define $\hat\tau_j$ as the infimum over
times $t>\tau_j$ such that
\[
\Delta_t(z) \le\tfrac{1}{2}\Delta_{\tau_j}(z)
\quad\mbox{or}\quad \Delta
_t(w) \le
\tfrac{1}{2}\Delta_{\tau_j}(w).
\]
Given this, let $\tau_{j+1}$ be the infimum over times $t>\hat\tau_{j}$
such that $\gamma(t) \in\overline{I}_{\hat\tau_j}$. These times are
understood to be infinite when past $T$, and hence at least one of the
points can no longer be approached by the curve. The sequence of
stopping times $\{\tau_k\}_{k \ge0}$ are called renewal times. We let
$R_{k+1} = 0$ if $\tau_{k+1} < \infty$ and $\Delta_{\tau_{k+1}}(z)
\leq
\frac12 \Delta_{\tau_k}(z)$; in this case, we can see that $\Delta
_{\tau_{k+1}}(w) > \frac12 \Delta_{\tau_k}(w)$. If $\tau_{k+1} <
\infty
$ and $\Delta_{\tau_{k+1}}(w) \leq\frac12 \Delta_{\tau_k}(w)$, we set
$R_{k+1} = 1$. We set $R_{k+1} = \infty$ if $\tau_{k+1} = \infty$. Less
formally, the renewal times encode when our curve halved its distance
to either $z$ or $w$ and then returned to $I_t$, while $R_k$ specifies
which point we halved the distance to. Let $\F_k = \F_{\tau_k}$.
\begin{lemma}\label{MainEstimate1}
There exist $c<\infty,\alpha>0$ such that for all $k \geq0$,
$r\leq1/2$,
\[
\Prob\{
R_{k+1} = 0; \Delta_{\tau_{k+1}}(z) \leq r
\Delta_{\tau_k}(z) \mid\F_k\}\leq
c 1\{\tau_k < T\}
\Delta_{\tau_k}(z)^\alpha
r^{2-d}.
\]
\end{lemma}
\begin{pf}
We assume $\tau_k < T$, and we write $\tau= \tau_k$, $\xi= \xi
(z;r\Delta_{\tau}(z))$. First, consider the event
that either
$I_\tau$ is not bounded, or both $I_\tau$ and $H_\tau^z$ are bounded.
By Lemmas~\ref{boundedLemma} and~\ref{oneSidedInfinite}, we have
$\Delta_\tau^*(z) \ge1$. Thus by Lemma~\ref{july173}, we get
\[
\Prob\{\xi< \infty\mid\F_k\} \leq c r^{2-d} \Delta_{\tau
}(z)^{\beta/2}.
\]
%

Suppose that $I_{\tau}$ is bounded,
and $H_{\tau}^w$ is bounded. We split into the following two cases:
$\Delta_\tau^*(z) \leq\sqrt{\Delta_\tau(z)}$ and $\Delta_{\tau
}^*(z) >
\sqrt{\Delta_\tau(z)}$. If $\Delta_{\tau}^*(z) > \sqrt{\Delta
_\tau
(z)}$, then
Lem\-ma~\ref{july173} implies
\[
\Prob\{\xi< \infty\mid\F_k\} \leq c r^{2-d} \Delta_{\tau
}(z)^{\beta/4}.
\]

Suppose $\Delta_{\tau}^*(z) \leq\sqrt{\Delta_\tau(z)}$. Then
there exist
simple curves $\eta_1,\eta_2\dvtx[0,1)
\rightarrow H_\tau^z$ contained in the disk of radius $2\Delta_\tau
^*(z)$ about $z$
with $\eta^j(0) = z$ and $\eta^j(1+) \in\partial_j H_\tau$.
At the time $\xi$ we can consider the line segment $L$ from $\gamma
(\xi
)$ to $z$. There exists a crosscut
of $H_\xi$,
$\hat\eta$,
contained in $L \cup\eta_1$ or in $L\cup\eta_2$, one of whose
endpoints is~$\gamma(\xi)$,
that disconnects $I_{\xi}$ from infinity. Using Lemma~\ref
{jul16lemma1}, we see
that
\[
\exc_{H_{\xi}}(\hat\eta,I_{\xi}) \leq c \Delta_\tau^*(z)^
{1/2} \leq c \Delta_\tau(z)^{1/4}.
\]
Thus, using Lemma~\ref{FullExcursion} we see that
\[
\Prob\{\xi< \tau_{k+1} < \infty\mid\F_k\}
\leq c \Delta_\tau(z)^{\beta/4}
\Prob\{\xi< \infty\mid\F_k \}
\leq c r^{2-d} \Delta_\tau(z)^{\beta/4} .
\]
\upqed\end{pf}
\begin{remark*}
The proof of the last lemma was not difficult given the estimates we have
derived. However, it is useful to summarize the basic idea. If $\Delta
_{\tau}^*(z)$
is not too small,
then it suffices to estimate
\[
\Prob\{
R_{k+1} = 0; \Delta_{\tau_{k+1}}(z) \leq r
\Delta_{\tau_k}(z) \mid\F_k\}
\]
by
\[
\Prob\{\xi< \infty\mid\F_k \}.
\]
However,
if $\Delta_{\tau_k}^*(z)$ is not much bigger than $\Delta_{\tau_k}(z)$
this estimate
is not sufficient. In this case, we need to use
\[
\Prob\{\xi< \infty\mid\F_k \}
\Prob\{\tau_{k+1} < \infty\mid\F_k, \xi< \infty\} .
\]
\end{remark*}

The above argument provides a good bound on the probability that the
near side gets even closer. To complete our argument, we must also
provide a bound limiting the probability that the far side can get
closer as well.
\begin{lemma}\label{MainEstimate2}
There exists $c<\infty$ such that for all $k \geq0$, $s \leq1/4$, if
\[
\xi^* = \inf\{ t > \tau_k | \Delta_t^*(z) \le s\} \quad\mbox{and}\quad
\eta^* = \inf\{ t > \xi^* | \gamma(t) \in I_{t^-} \},
\]
then
\[
\Prob\{\eta^* < \infty, \Delta_{\eta^*}^*(z) \le s | \Delta_{\tau
_k}^*(z) > s, \F_{\tau_k}\} \le c s^{\beta/2}.
\]
\end{lemma}
\begin{pf}
Assume $\Delta_{\tau_k}^*(z) > s$.
If
$\eta^* < \infty$ we may define
\[
\varpi= \sup\{ t < \eta^* | \gamma(t) \in I_{t^-}\}\vadjust{\goodbreak}
\]
to be the previous time that $\gamma$ crossed $I_{t^-}$ before $\eta
^*$. Note that $\tau_k \leq
\varpi< \xi^* < \eta^*$ and $\Delta_{\varpi}^*(z) > s$.
By considering the two times $\varpi$ and $\eta^*$ in Lemma~\ref
{boundLemma}, we see that $H_{\varpi}^w$ is bounded.

Consider the situation at time $\xi^*$. By the definition
of the stopping times, there must be a curve $\nu\dvtx (0,1) \rightarrow
H_{\xi^*}$ which contains $z$, is never more than distance
$2s$ from $z$, has $\nu(0^+) \in\partial_1H_{\xi^*}$ and $\nu(1^-)
\in\partial_2H_{\xi^*}$ such that $\nu$ separates $I_{\xi^*}$, and
hence~$w$, from infinity. Since $\nu$ is at least distance $1/2$ from
$I_{\xi^*}$ we know
from Lemma~\ref{jul16lemma1}
that the excursion measure between $\nu$ and $I_{\xi^*}$ in $H_{\xi
^*}$ is bounded above by
$C s^{1/2}$.
Then an application of Lemma~\ref{FullExcursion} tells us that the
probability of $\gamma$ returning to $I_{\xi^*}$ is bounded above by
$Cs^{\beta/2}$ which gives the lemma.
\end{pf}

The following two lemmas combine the methods of the above two bounds.

\begin{lemma}\label{MainEstimate3}
There exist $c<\infty,\alpha>0$ such that for all $k \geq0$,
$r\leq1/2$, $s \le1/4$,
\begin{eqnarray*}
&&\Prob\{
R_{\tau_{k+1}} = 0; \Delta_{\tau_{k+1}}(z) \leq r
\Delta_{\tau_k}(z) ; \Delta_{\tau_{k+1}}^*(w)
\leq s \mid\F_k\} \\
&&\qquad \le c 1\{\tau_k < T\}
\Delta_{\tau_k}(z)
^{\alpha} [s^\alpha+ 1\{\Delta_{\tau_k}^*(w)
\leq s \} ] r^{2-d}.
\end{eqnarray*}
\end{lemma}
\begin{pf}
If $\Delta_{\tau_k}^*(w) \le s$, then the desired statement reduces to
Lemma~\ref{MainEstimate1}. Thus, we may assume that $\Delta_{\tau
_k}^*(w) > s$.

Let $\zeta^* =
\zeta^*_k$ be the infimum over times $t > \tau_k$ so that $\Delta
^*_t(w) \le s$ and $\gamma(t) \in I_{t^-}$.
Let $\sigma= \sigma_k = \inf\{t > \tau_k | \Delta_t(z) \le r\Delta
_{\tau_k}(z)\}$. If
$\Delta_{\tau_k}^*(w) > s, \Delta_{\tau_{k+1}}^*(w)\leq s$, and
\mbox{$\sigma< \infty$}, then
$\zeta^* < \sigma$ since the curve $\gamma$ would need to intersect
$I_\sigma$ before approaching $w$ and hence would force the renewal
time $\tau_{k+1}$ before $\zeta_k$.

By the same argument as in Lemma~\ref{MainEstimate2}, we know if
$\Delta
_{\tau_k}^*(w) > s$
and $\zeta^* < \infty$, there is a time $\omega$,
$\tau_k \le\omega< \zeta^*$ for which there is a curve connecting
$\partial_1H_\omega$ to $\partial_2H_\omega$ passing through
$\gamma
(\omega)$ contained in a disk of radius $2s$ about $w$ separating
$I_{\kappa}$ from infinity. Then, by Lemma~\ref{FullExcursion}, we
have that
\[
\Prob\{\zeta^* < \infty| \Delta^*_{\tau_k}(z) > s, \F_{\tau_k}\}
\le
cs^\alpha.
\]

By Lemma~\ref{boundLemma} we know $H_{\zeta^*}^z$ is bounded.
Lemma~\ref{oneSidedInfinite} implies that
\mbox{$\Delta^*_{\zeta^*}(z) = 1$}, and hence by Lemma~\ref{jul16lemma1}
\[
\Prob\{R_{\tau_{k+1}} = 0; \Delta_{\tau_{k+1}} \le r \Delta_{\tau_k}(z)
| \F_{\zeta^*}, \zeta^* < \infty\} \le c1\{\tau_k < T\}\Delta
_{\zeta
^*}(z)^\alpha r^{2-d}.
\]
Combining the above two bounds gives the desired result.
\end{pf}
\begin{lemma}\label{MainEstimate4}
There exist $c<\infty,\alpha>0$ such that for all $k \geq0$,
$r\leq1/2$, $s > 0$,
\begin{eqnarray*}
&&\Prob\{
R_{\tau_{k+1}} = 0; \Delta_{\tau_{k+1}}(z) \leq r
\Delta_{\tau_k}(z) ; \Delta_{\tau_{k+1}}^*(z)
\leq s \mid\F_k\}\\
&&\qquad \le c 1\{\tau_k < T\}
\Delta_{\tau_k}(z)
^{\alpha} [s^\alpha+ 1\{\Delta_{\tau_k}^*(z)
\leq s \} ] r^{2-d}.
\end{eqnarray*}
\end{lemma}
\begin{pf}
If $\Delta_{\tau_k}^*(z) \le s$ or $s \geq1/4$,
the conclusion
reduces to Lemma~\ref{MainEstimate1}. Thus we may assume that $\Delta
_{\tau_k}^*(z) > s, s \leq1/4$. Let $E$ denote the event
\[
E = \{ R_{\tau_{k+1}} = 0; \Delta_{\tau_{k+1}}(z) \leq r
\Delta_{\tau_k}(z) ; \Delta_{\tau_{k+1}}^*(z)
\leq s ; \Delta_{\tau_{k}}^*(z)
> s \}.
\]

Let
\[
\sigma= \inf\{ t | \Delta_t(z) \le r\Delta_{\tau_k}(z)\},
\]
and note that on the event $E$,
\[
\tau_{k+1} = \inf\{t > \sigma| \gamma(t) \in I_{t^-}\}.
\]
Define $\xi$ to be the infimum over times $t\ge\sigma$ such that there
is a curve $\eta\dvtx (0,1) \rightarrow H_t$ with $\eta(0^+) = \gamma(t)$
and $\eta(1^-) \in\partial H_t$ with
$\eta$ contained entirely in the ball of radius $2s$ about $z$, and
$\eta$ separating $I_t$ from $\infty$.

We now claim that given $\F_\sigma$ either $\xi< \tau_{k+1}$ or
$\Delta^*_{\tau_{k+1}}(z) > s$. To see this, suppose neither holds.
Since $\Delta^*_{\tau_{k+1}}(z) \le s$, for every $s<s' \leq2s
\leq1/2$, there is a crosscut $\eta$ of $H_{\tau_{k+1}}$ going through
$z$ whose endpoints are in $ \partial_1 H_{\tau_{k+1}}, \partial_2
H_{\tau_{k+1}}$, respectively, and which is contained in the disk of
radius $s'$ about $z$. By Lemma~\ref{boundLemma} we
know $\eta$ must disconnect $I_{\tau_{k+1}}$ from $\infty$ since
$H_{\tau _{k+1}}^w$ must be bounded. We can choose such an $\eta$ such
that at least one endpoint of $\eta$ is not in $\gamma[0,\tau_{k}]$,
for otherwise all such $\eta$ would be a crosscuts of $H_{\tau_{k}}$
separating $w$ from infinity which would imply that
$\Delta_{\tau_k}^*(z) \leq s$.

Let $\zeta= \sup\{t\le\tau_{k+1} | \gamma(t) \in\overline{\eta}
\} >
\tau_k$ and note
that $\tau_k < \zeta< \tau_{k+1}$.
If $\zeta\ge\sigma$ we are done since this $\eta$ demonstrates that
$\xi< \tau_{k+1}$.

Thus assume $\zeta< \sigma$. In this case, we will construct a curve
in $H_\sigma$ satisfying the conditions in the definition of $\xi$.
Since $\zeta< \sigma$ we know the curve $\eta$ defined above
disconnects $I_\sigma$ from infinity in $H_\sigma$. By the definition
of $\sigma$ as the first time that $\Delta_\sigma(z) \le r\Delta
_{\tau
_k}(z)$, the straight open line segment, $L$, from $\gamma(\sigma)$ to
$z$ is contained~in~$H_\sigma$. Additionally, since $\Delta_\sigma(z)
\le\Delta^*_\sigma(z) \le s$, we know $\eta(0,1) \cup L$ is contained
entirely in the ball of radius $2s$ about $z$. Thus we may find a curve
$\hat\eta$ contained in $\eta(0,1) \cup L$ which separates $I_\sigma$
from infinity in $H_\sigma$ with $\eta(0^+) = \gamma(t)$ and $\eta(1^-)
\in\partial H_t$ and with $\eta$ contained entirely in the ball of
radius $2s$ about $z$, proving that $\xi= \sigma< \tau_{k+1}$. Thus
we have reached a contradiction.

On the event $E$
we know $\Delta^*_{\tau_{k+1}}(z) \le s$, and thus the above argument
tells us $\xi< \tau_{k+1}$.
We have therefore shown that if $\Delta_{\tau_k}^*(z) > s, s \leq
1/4$, then
\[
\Prob(E \mid\F_k) \leq\Prob\{\sigma\leq\xi< \tau_{k+1} <
\infty
\mid\F_k\}.
\]
We may now argue as in the second part of the proof of Lemma \ref
{MainEstimate1} to obtain
$\Prob\{\sigma< \infty\mid\F_k\} \leq c \Delta_{\tau
_k}(z)^\alpha$ and
$\Prob\{\tau_{k+1} < \infty\mid\F_\xi\} \leq c s^\alpha$.
%
\end{pf}

\subsection{Combinatorial estimates}

We have now completed the bulk of the probabilistic estimates. Most of
what remains is a combinatorial argument to sum up the bounds proven
above across all possible ways that the $\mbox{SLE}$ curve may approach $z$
and $w$ in turn.\vadjust{\goodbreak}

Without loss of generality, assume that $\delta= 2^{-m}$ and $\epsilon
= 2^{-n}$,
and let
\begin{eqnarray*}
\xi_z &=& \xi_{z,\epsilon} = \inf\{t\dvtx \Delta_t(z)
\leq2^{-n}\},\\
\xi_w &=& \xi_{w,\delta} = \inf\{t\dvtx \Delta_t(w) \leq2^{-m} \},
\\
\xi &=& \xi_z \vee\xi_w = \inf\{t\dvtx \Delta_t(z) \leq2^{-n},
\Delta_t(w) \leq2^{-m} \} .
\end{eqnarray*}
These are similar to $\chi$ and $\xi$ from the previous sections;
however now the times denote the first time that the curve gets within
a small Euclidean distance of the point, rather than a small conformal
radius of the point. Let $\sigma$ be the minimal $\tau_k$ such that
$\Delta_{\tau_k}(z) < 2^{-n+1}$ or $\Delta_{\tau_k}(w) < 2^{-m+1}$. Let
$k_\sigma$ be the index so that $\sigma= \tau_{k_\sigma}$. If such a
renewal time does not exist, let $k_\sigma= \infty$ and $\sigma=
\infty$. Note that if $\xi$ is finite, then so is $\sigma$.

Let $V_{z,k},V_z$ denote the events (and their indicator functions)
\[
V_{z,k} = \{ k_\sigma= k, R_\sigma=0\},\qquad V_z = \bigcup_{k=1}^\infty
V_{z,k}.
\]
We define $V_w$ analogously.
By the definition of $\sigma$, on the event
the event $V_z$,
\[
\Delta_{\tau_{k_\sigma-1}}(z) \geq2^{-n+1},\qquad
\Delta_{\tau_{k_\sigma-1}}(w) \geq2^{-m+1} ,\qquad
\Delta_\sigma(z) < 2^{-n+1} .
\]
Also,
\[
\Delta_\sigma(w) > 2^{-m} ,
\]
for if $\Delta_\sigma(w) \leq2^{-m}$, there would have been a
renewal time
after $\tau_{k_\sigma- 1}$ but before $\tau_k = \sigma$.
Note that
\[
\{\xi< \infty\} \subset[V_z \cap\{\xi_w < \infty\}]
\cup[V_w \cap\{\xi_z < \infty\} ].
\]
We will concentrate on the event $V_z \cap\{\xi_w < \infty\}$;
similar arguments handle the event $V_w \cap\{\xi_z < \infty\} $.

Define the integers $(i_l,j_l)$ by stating that at the
renewal time $\tau_l$,
\[
2^{-i_l} < \Delta_{\tau_l}(z)
\leq2^{-i_l + 1},\qquad
2^{-j_l} < \Delta_{\tau_l}(w) \leq
2^{-j_l+1} .
\]
If $\sigma< \infty$, we write $(i_\sigma,j_\sigma) =
(i_{k_{\sigma}},j_{k_\sigma})$.
On the event $k_\sigma= k, R_\sigma= 0$, there
is a finite sequence of ordered
triples
\begin{eqnarray}
\pi= [(i_0,j_0,0), (i_1,j_1,R_1), \ldots, (i_{k-1},j_{k-1},R_{k-1})
,(i_k,j_k,R_k) =(i_\sigma,j_\sigma,0)],
\nonumber\\
&&\eqntext{i_l,j_l \in\{1,2,3,\ldots\},
R_l \in\{0,1\}.}
\end{eqnarray}
We have the following properties for
$0 \leq l \leq k-1$:
\begin{itemize}
\item If $R_{l+1} = 0$, then
$i_{l+1} \geq i_l + 1$
and $j_l \leq j_{l+1} \leq j_l
+1$.
\item If $R_{l+1} = 1$, then
$i_l \leq i_{l+1} \leq i_l + 1$ and $j_{l+1} \geq j_l + 1$.
\end{itemize}
We call any sequence of triples satisfying these two properties a
\textit{legal}
sequence of length $k$. For any $i,j,k$, let $\seq_k(i,j,0)$ denote the
collection of legal finite sequences of length $k$ whose final triple is
\[
(i_k,j_k,R_k) = (i,j,0).\vadjust{\goodbreak}
\]
If $\pi$ is a legal finite sequence of length $k$, let $V_{z,\pi}$ be
the event that $k_\sigma= k$, \mbox{$R_\sigma= 0$} and the renewal times up to
and including $\sigma$ give the sequence $\pi$. Figure~\ref{combEx}
illustrates this definition.

%
\begin{figure}

\includegraphics{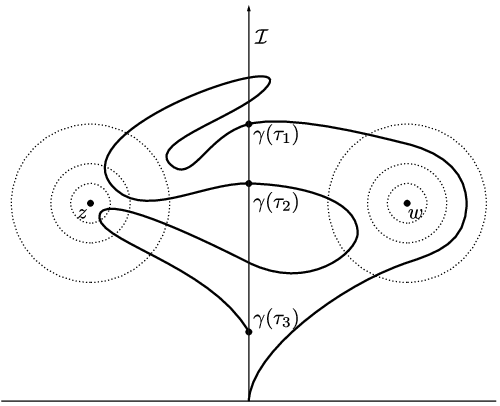}\vspace*{-2pt}

\caption{A curve $\gamma$ (shown in bold) in $V_{z,\pi}$ with $\pi=
[(0,0,0), (0,1,1), (2,1,0), (3,1,0)]$.}\vspace*{-3pt}
\label{combEx}
\end{figure}

Define $K_l$ for $1\leq l \leq k$ by
\[
K_l = \cases{
i_{l-1}, &\quad if $R_l = 0$, \cr
j_{l-1}, &\quad if $R_l = 1$.}
\]

The next proposition gives the fundamental estimate.\vspace*{-2pt}
\begin{proposition} \label{fundest}
There exist $c$ and an $\alpha>0$ such that
the following holds. Let $i,j,k$ be integers, and let
$\pi\in\seq_k(i,j,0)$. Then
\[
\Prob[V_{z,\pi} \cap\{\xi_w< \infty\}]\le
c^k 2^{(m+n)(d-2)}
e^{-\alpha(i+j-n)}
\prod_{l=1}^k e^{-\alpha K_l }.\vspace*{-3pt}
\]
\end{proposition}
\begin{pf}
Note that on the event $V_z$ we may say by Lemma~\ref{july173} that
\[
\Prob\{\xi_w< \infty| \F_k \}
\leq c \biggl[\frac{2^{-j}}{\Delta^*_{\tau_k}(w)}\biggr]^{\beta/2} 2^{(m-j)(d-2)}.
\]
We will proceed by splitting the event $V_{z,\pi}$ into the case where
$\Delta^*_{\tau_k}(w) \ge2^{-j}$ and the case where it is not.

Before doing so, we will discuss how to estimate $\Prob[V_{z,\pi}]$
without any further conditions, as it is important to our bounds below.
By the definition of $\seq_k(i,j,0)$ we have that
\[
\pi= [(i_0,j_0,0), (i_1,j_1,R_1), \ldots, (i_k,j_k,0)],
\]
where the sequence of triples is a legal sequence as described above.\vadjust{\goodbreak}

We will estimate this probability by applying Lemma~\ref{MainEstimate1}
to approximate the probability of each step; which is to say the
probability that given that the $\mbox{SLE}$ at time $\tau_l$ yields the
triple $(i_l, j_l, R_l)$ we get the triple $(i_{l+1}, j_{l+1},
R_{l+1})$ at the time~$\tau_{l+1}$. As the two cases are similar, we
assume that $R_{l+1} = 0$. Since $K_{l+1} = i_l$, we know the distance
to $z$ at time $\tau_l$ is less than $2^{-K_{l+1}}$. We wish the
distance from $z$ to the $\mbox{SLE}$ curve to decrease by at least a factor
of $2^{i_{l+1}-i_l-1}$. The probability of this is shown by Lemma~\ref
{MainEstimate1} to be of the order $c2^{-\alpha
K_{l+1}}2^{(2-d)(i_{l+1}-i_l)}$ by absorbing factors into $\alpha$ and
$c$ we may rewrite this bound as $c e^{-\alpha K_{l+1}}2^{(d-2)(i_{l+1}
- i_{l} + j_{l+1} - j_{l})}$ as $j_{l+1}$ can be at most one greater
than~$j_{l}$.

To get the probability of $V_{z,\pi}$, we need only multiply through
each of these $k$ individual probabilities to get that
\begin{eqnarray*}
\Prob[V_{z,\pi}] & \le & \prod_{l = 1}^k c e^{-\alpha
K_{l}}2^{(d-2)(i_{l} - i_{l-1} + j_{l} - j_{l-1})} \\
& = & c^k \exp\Biggl\{ \log(2) (d-2)\sum_{l = 1}^k (i_{l} - i_{l-1} + j_{l} -
j_{l-1})\Biggr\} \prod_{l = 1}^k e^{-\alpha K_{l}} \\
& = & c^k 2^{(d-2)(i+j)} \prod_{l = 1}^k e^{-\alpha K_{l}},
\end{eqnarray*}
where we have absorbed the $2^{(d-2)(i_0+j_0)}$ (bounded above by a
constant given the restrictions of $z$ and $w$) into the $c^k$ term in
the last line by redefining $c$.

We now return to our main estimate. Note that, when $\Delta^*_{\tau
_k}(w) \ge2^{-j/2}$, we have
\begin{eqnarray*}
&&\Prob[V_{z,\pi} \cap\{\Delta^*_{\tau_k}(w) \ge2^{-j/2}\} \cap\{
\xi
_w< \infty\}]\\
&&\qquad \le c \Prob[V_{z,\pi} \cap\{\Delta^*_{\tau_k}(w) \ge2^{-j/2}\}]
2^{-\beta j/4} 2^{(m-j)(d-2)} \\
&&\qquad \le c \Prob[V_{z,\pi}] 2^{-\beta j/4} 2^{(m-j)(d-2)} \\
&&\qquad \le c^k 2^{-\beta j/4} 2^{(m-j)(d-2)} 2^{(i+j)(d-2)} \prod_{l=1}^k
e^{-\alpha K_l } \\
&&\qquad = c^k 2^{-\beta j/4} 2^{(m+n)(d-2)} 2^{(i-n)(d-2)}\prod_{l=1}^k
e^{-\alpha K_l } \\
&&\qquad \le c^k 2^{(m+n)(d-2)} 2^{-\mu(i+j-n)} \prod_{l=1}^k e^{-\alpha K_l },
\end{eqnarray*}
where the third line follows from the above discussion, and the last
line holds for some choice of $\mu> 0$.

Thus we need only understand the event
\begin{eqnarray*}
&&\Prob[V_{z,\pi} \cap\{\Delta^*_{\tau_k}(w) < 2^{-j/2}\} \cap\{\xi_w<
\infty\}] \\
&&\qquad \le c \Prob[V_{z,\pi} \cap\{\Delta^*_{\tau_k}(w) < 2^{-j/2}\}]
2^{(m-j)(d-2)}.
\end{eqnarray*}
For the event $\{\Delta^*_{\tau_k}(w) < 2^{-j/2}\}$ there must be at
least one $l$ such that $\Delta^*_{\tau_l}(w) \ge2^{-j/2}$ and
$\Delta
^*_{\tau_{l+1}}(w) < 2^{-j/2}$. By using Lemma~\ref{MainEstimate3} for
that single step if $R_l =1$ or Lemma~\ref{MainEstimate4} if $R_l = 0$
and~\ref{MainEstimate1} for all other steps, we have that
\begin{eqnarray*}
&&\Prob[V_{z,\pi} \cap\{\Delta^*_{\tau_k}(w) < 2^{-j/2}\}]\\
&&\qquad \le\sum_{l = 0}^{k-1} \Prob[V_{z,\pi} \cap\{\Delta^*_{\tau_l}(w)
\ge2^{-j/2}; \Delta^*_{\tau_{l+1}}(w) < 2^{-j/2}\}] \\
&&\qquad \le kc^k 2^{-\alpha j/2} 2^{(i+j)(d-2)} \prod_{l=1}^k e^{-\alpha K_l }.
\end{eqnarray*}
By combining this with the above event we see that
\begin{eqnarray*}
&&\Prob[V_{z,\pi} \cap\{\Delta^*_{\tau_k}(w) < 2^{-j/2}\} \cap\{\xi_w<
\infty\}] \\
&&\qquad \le kc^k 2^{(m-j)(d-2)} 2^{-\alpha j/2} 2^{(i+j)(d-2)} \prod_{l=1}^k
e^{-\alpha K_l } \\
&&\qquad \le c^k 2^{(m+n)(d-2)} 2^{-\mu(j+i-n)} \prod_{l=1}^k e^{-\alpha K_l }
\end{eqnarray*}
for some choice of $\mu$ and where $c$ is being used generically to
absorb the leading $k$. Thus by choosing $\mu$ and $\alpha$ to be the
same (which we can do by taking the minimum for both) we get the
desired result.
\end{pf}

We will now show how this proposition implies the main theorem. The
proof rests upon the following combinatorial lemma.
\begin{lemma}\label{comblem} For every $\alpha>0$, there exist $c$ and
a $u>0$ such that for all $k$
\[
\sum_{\pi\in\seq_k(i,j,0)}
\prod_{l=1}^k e^{-\alpha K_l }
\leq c e^{-u k^2}.
\]
\end{lemma}
\begin{pf} We fix $\alpha$ and allow all constants
to depend on $\alpha$.
Let
\[
\Sigma_k = \sum_{[m]_k} \prod_{l=1}^k e^{-\alpha m_l },
\]
where the sum is over all strictly increasing finite sequences of
positive integers, written as $[m]_k := [m_1, m_2, \ldots, m_k]$. We
first claim that
\[
\Sigma_k \leq c_1 e^{- \alpha k^2/4}.
\]
Consider the following recursive relation:
\begin{eqnarray*}
\Sigma_k & = & \sum_{[m]_k} \prod_{l=1}^k e^{-\alpha m_l } \\
& \le & \sum_{[m]_{k-1}}\sum_{m_k = k}^\infty e^{-\alpha m_k} \prod
_{l=1}^{k-1} e^{-\alpha m_l } \\
& = & \Sigma_{k-1} \sum_{j = k}^\infty e^{-\alpha j} \\
& \le & c_2 \Sigma_{k-1} e^{-\alpha k}.
\end{eqnarray*}
Applying this bound inductively to $\Sigma_k$ yields
\[
\Sigma_k \le c_2^k \exp\Biggl\{-\alpha\sum_{i = 1}^k i\Biggr\}
\le c_1 e^{-\alpha k^2/4}
\]
as desired.

To choose a legal sequence in $\seq_k(i,j,0)$, there are $2^{k-1}$
ways to choose the values $R_1,\ldots,R_{k-1}$. Given the
values of $R_1,\ldots,R_{k-1}$ we choose the increases of
the integers. If $R_l = 0$, then $i_l > i_{l-1}$ and
$j_l = j_{l-1}$ or $j_l = j_{l-1} + 1$. The analogous
inequalities hold if $R_1 = 1$. There are $2^k$ ways to choose
whether $j_l = j_{l-1}$ or $j_l = j_{l-1} + 1$ (or the
corresponding jump for $i_l$ if $R_1 = 1$). In the other
components we have to increase by an integer. We therefore
get that the sum is bounded above by
\begin{eqnarray*}
2^{k-1} \max_{0\le l\le k-1} 2^l\Sigma_l \cdot2^{k-l-1}\Sigma_{k-l-1}
& \le & c^k \max_{0 \le l \le k-1} e^{- \alpha l^2/4} e^{- \alpha
(k-l-1)^2/4} \\
& \le & c e^{-u k^2}.
\end{eqnarray*}
%
\upqed\end{pf}

By combining Proposition~\ref{fundest} and Lemma~\ref{comblem}, there
exist $c$ such that
\[
\sum_{k=1}^\infty
\sum_{\pi\in\seq_k(i,j,0)} \Prob[V_{z,\pi}\cap\{\xi_w< \infty\}
]
\leq c 2^{(m+n)(d-2)}
e^{-\alpha(j + i-n)},
\]
and hence by summing over $i \geq n-1,j\geq0$ we get
\begin{eqnarray*}
\Prob[V_z \cap\{\xi< \infty\}] &\le& \Prob[V_z \cap\{\xi_w <
\infty\}]
\\
& = &
\sum_{i = n-1}^\infty\sum_{j = 0}^\infty\sum_{k=1}^\infty\sum
_{\pi
\in\seq_k(i,j,0)}
\Prob[V_{z,\pi} \cap\{\xi_w < \infty\} ] \\
&\le &c 2^{(m+n)(d-2)} = c \epsilon^{2-d} \delta^{2-d}.
\end{eqnarray*}
By the symmetry of $z,w$ we have the bound
\[
\Prob[V_w \cap\{\xi< \infty\}]\le c \epsilon^{2-d} \delta^{2-d}
\]
and hence
\[
\Prob\{\Delta_\infty(z) \leq\epsilon,
\Delta_\infty(w) \leq\delta\} = \Prob\{\xi< \infty\} \le c
\epsilon
^{2-d} \delta^{2-d}
\]
as required to complete the proof of Proposition~\ref{main2}, and hence
the proof of Beffara's estimate.

With the proof set up in this way, we may now rapidly complete our
proof of the existence of the multi-point Green's function. By mirroring
the proof above, we may conclude that for $\rho= e^{-\ell}$ (and hence
for all $\rho$) that
\begin{eqnarray*}
\Prob[V_z \cap\{\xi< \infty, \Delta_\sigma(w) \leq\rho\}] &\le&
\Prob[V_z \cap\{\xi_w < \infty, \Delta_\sigma(w) \leq\rho\}] \\
& = &
\sum_{i = n-1}^\infty\sum_{j = \ell}^\infty\sum_{k=1}^\infty\sum
_{\pi\in\seq_k(i,j,0)}
\Prob[V_{z,\pi} \cap\{\xi_w < \infty\} ] \\
&\le& c 2^{(m+n)(d-2)} e^{-\alpha\ell}= c \epsilon^{2-d} \delta^{2-d}
\rho^\alpha.
\end{eqnarray*}

This proves Proposition~\ref{fromBef} and hence completes the proof of
the existence of the multi-point Green's function.

\begin{appendix}\label{app}
\section{The existence of the $I_t$}\label{ItAppendix}

The aim of this Appendix is to prove the existence of the separating
set $I_t$ desired above.
\begin{definition*}
Let $\gamma$ be a curve in the upper half-plane, and let $z,w,\imax$ be
a pair or distinct points in $\Half$ separated by the line $\imax$. Let
$\imax_t = \imax\setminus\gamma(0,t]$. We will denote by $I_t$ the
unique open interval contained in $\imax$ such that the following four
properties hold. For any $t \le t'$ we have:
\begin{itemize}
\item$I_t$ is a connected component of $\imax_t$,
\item the $I_t$ are decreasing, which is to say $I_{t'} \subseteq I_t$,
\item$H_t \setminus I_t$ has exactly two connected components, one
containing $z$ and one containing $w$ and
\item$I_t = I_{t'}$ whenever $\gamma(t,t'] \cap\imax= \varnothing$.
\end{itemize}
\end{definition*}

It may, at first glance, seem simple to define such sets inductively.
However, in general, the set of times that a curve $\gamma$ crosses
$\imax$ may be uncountable and have no well-defined notion of ``the
previous crossing.'' To avoid this issue and show this notion is well
defined, we require a few topological lemmas.\vadjust{\goodbreak}

\begin{lemma}
Let $U$ be a connected open set in $\C$ separated by a smooth simple
curve $\eta\dvtx [0,1] \rightarrow\overline U$. Let $V \subset U$ be a
connected open subset. Then for any points $z,w \in V$, there exits a
curve $\xi\dvtx [0,1] \rightarrow V$ from $z$ to $w$ which intersects $\eta$
a finite number of times.
\end{lemma}
\begin{pf}
This proof mirrors the classic proof that a connected open set is path
connected. Define an equivalence relation on $V$ where points $z,w \in
V$ are equivalent, if $z$ can be connected to $w$ by a curve $\xi$
which intersects $\eta$ a finite number of times. This can readily be
shown to satisfy the requirements of an equivalence relation.

Let $V_\alpha$ denote the open connected components of $V \setminus
\eta
$. If $z,w$ are both in the same $V_\alpha$, then they may be connected
by a curve which does not intersect $\eta$; hence each $V_\alpha$ is
contained entirely in a single equivalence class.

Consider a disc, $D$, contained in $V$ centered on a point $\eta(t_0)$
for some $t_0 \in(0,1)$ with components $V_\alpha$ and $V_\beta$ on
either side of $\eta$ near this point. Since $\eta$ is smooth and
simple, by choosing $D$ sufficiently small we may find a diffeomorphism
$\phi$ so that $\phi(D) = \Disk$ and $\phi(\eta\cap D) = \{it\dvtx t
\in
(-1,1)\}$. Connect $-1/2$ to $1/2$ by the straight line between them,
which only intersects the image of $\eta$ once. Taking the image of
this line under $\phi^{-1}$ gives a curve $\xi$ satisfying the
conditions of the equivalence relation connecting two points, one in
$V_\alpha$ and one in $V_\beta$. Thus components of $V \setminus\eta$
which are directly separated by $\eta$ are in the same equivalence
class. Since $V$ is connected, the only equivalence class is $V$ itself.
\end{pf}

Suppose $U$ is a connected open set in $\C$ separated by a curve $\eta
\dvtx(0,1)\rightarrow U$ into two components $U_1,U_2$ with points $z\in
U_1$ and $w \in U_2$. Let $V$ be a connected subset of $U$. Define
$\mathcal{D}_V(z,w;\eta)$ to be the the set of connected components of
$V \cap\eta$ which disconnects $z$ from $w$ in $V$.
\begin{corollary} \label{oddCor}
Let $U$ be a connected open set in $\C$ separated by a smooth simple
curve $\eta\dvtx [0,1] \rightarrow\overline U$ into two components
$U_1,U_2$ with $z \in U_1$ and $w \in U_2$. Let $V \subset U$ be a
connected open subset containing $z$ and $w$. Then $|\mathcal
{D}_V(z,w,\eta)|$ is finite and odd.
\end{corollary}
\begin{pf}
To see that the number is finite, take the curve $\xi$ between $z$ and
$w$ as in the above lemma, and note that any $\eta_i$ which separates
$z$ from $w$ must intersect~$\xi$.

To see that it is odd, consider the connected components of $V' := V
\setminus\bigcup_{\gamma\in\mathcal{D}_V(z,w;\eta)} \gamma$. There
are exactly $|\mathcal{D}_V(z,w;\eta)| + 1$ such components. $\eta$
separates $U$ into two components, and hence the components of $V'$ are
alternately contained in $U_1$ and $U_2$. Since the component
containing $z$ is in $U_1$, and the component containing $w$ is in
$U_2$, there must be an even number of components of $V'$, which makes
$|\mathcal{D}_V(z,w;\eta)|$ odd.
\end{pf}

This general topological lemma has the following consequence in our
setting. To simplify notation, we will define $\mathcal{D}_t =
\mathcal
{D}_{H_t}(z,w,\imax)$.
\begin{corollary} \label{oddCor2}
Fix $0 \le t' \le t < \infty$. Then a connected component $I$ of
$\imax
_{t'}$ separates $z$ from $w$ in $H_{t'}$ if and only if the number of
elements of $\mathcal{D}_t$ contained in $I$ is odd.
\end{corollary}
\begin{pf}
The ``only if'' direction is precisely Corollary~\ref{oddCor}. Thus we
wish to show that if the number of elements of $\mathcal{D}_t$
contained in $I$ is odd, then $I$ separates $z$ from $w$.

Assume not, so the number of elements of $\mathcal{D}_t$ contained in
$I$ is odd but $I$ does not separate $z$ from $w$. $H_{t'} \setminus I$
has two components, one of which contains both $z$ and $w$. Consider
any curve $\eta$ connecting $z$ to $w$. Without loss of generality
assume that $\eta$ crosses each element of $\mathcal{D}_t$ exactly once
by simply removing any portion of the curve between the first and last
times that it crosses each element of $\mathcal{D}_t$. Since $\eta$
crosses each element of $\mathcal{D}_t$ contained in $I$ precisely
once, we know $\eta$ crosses $I$ an odd number of times, and hence it
must start and end in different components of $H_{t'}\setminus I$ which
contradicts the fact that it connects $z$ to $w$.
\end{pf}

We may now use this to prove that $I_t$ is well defined.
\begin{pf*}{Proof of well-definedness of $I_t$}
For a component $I$ of $\imax_t$ and $t' < t$, let $C_{t'}(I)$ denote
the component of $\imax_{t'}$ which contains $I$. We claim there exists
a unique component of $\imax_t$, which we will denote $I_t$, such that
for all $0 \le t' \le t$, we have $C_{t'}(I_t) \in\mathcal{D}_{t'}$.
Note that such an $I_t$ clearly satisfies all the conditions of the definition.

First we prove existence. Let $\{J_i\}_{i = 1}^\infty$ be the connected
components of $\imax_t$. Assume that none satisfy the above condition,
which is to say that for each $i$ there exists a $t_i \le t$ so that
$C_{t_i}(J_i)$ does not separate $z$ from $w$ in $H_{t_i}$. Now $\{
C_{t_i}(J_i)\}_{i = 1}^\infty$ covers $\imax_t$ since the $J_i$ did as
well, and moreover since by construction the $C_{t_i}(J_i)$ are either
contained in each other or disjoint, we may find a sub-collection $\{
C_{t_{i_k}}(J_{i_k})\}_{k = 1}^\infty$ which covers $\imax_t$ with all
elements pairwise disjoint. By Corollary~\ref{oddCor2} there are an
even number of elements of $\mathcal{D}_{t}$ contained in
$C_{t_{i_k}}(J_{i_k})$ for each $k$. However, since they cover
disjointly, this implies that $|\mathcal{D}_{t}|$ is even, which
contradicts Corollary~\ref{oddCor} completing the proof of existence.

Now we establish uniqueness. Let $I_t^{(1)}, I_t^{(2)}, \ldots,
I_t^{(\ell)}$ denote the components of $\imax_t$ such that for all $0
\le t' \le t$ we have $C_{t'}(I_t^{(i)}) \in\mathcal{D}_{t'}$, and
assume that $\ell> 1$. Define
\[
t_0 = \sup\bigl\{t' \dvtx \exists_{i \neq j}
\mbox{ s.t. } C_{t'}\bigl(I_t^{(i)}\bigr) =
C_{t'}\bigl(I_t^{(j)}\bigr)\bigr\}.
\]
By this definition, it is clear that $\gamma(t_0) \in\imax$. Moreover,
there exists a $t_1 < t_0$ such that $\gamma[t_1,t_0) \cap\imax=
\varnothing$ since if there did not then $\gamma(t_0)$ is a limit point\vadjust{\goodbreak}
of $\gamma(0$, $t_0)\cap\imax$ which implies that an earlier time would
have separated all the $I_t^{(i)}$ from each other contradicting the
choice of $t_0$. The components of $\imax_{t_0}$ are precisely those of
$\imax_{t_1}$ except for a single component, call it $J$, which is
split into $J_1,J_2$ in $\imax_{t_0}$ by $\gamma(t_0)$. By the choice
of $t_0$, $J$ is $C_{t_1}(I_t^{(i)})$ for some $i$ and both of
$J_1,J_2$ are $C_{t_0}(I_t^{(i)})$ for some $i$. This is a
contradiction since by Corollary~\ref{oddCor2} each of $J,J_1,J_2$ must
contain an odd number of elements of~$\mathcal{D}_{t}$.\looseness=-1~%
\end{pf*}

\section{The PDE for the Green's function}\label{PDEcomp}
We outline here the derivation of the PDE which governs the ordered
version of the multi-point Green's function. From Theorem
\ref{multitheorem}, we
know that for $z, w \in\Half$ with
\begin{eqnarray*}
\xi &=& \xi_\epsilon= \xi_{z,\epsilon} = \inf\{t\dvtx \Upsilon_t(z)
\leq
\epsilon\},
\\
\chi &=& \chi_\delta= \chi_{w,\delta} = \inf\{t\dvtx \Upsilon_t(w) \leq
\delta\},
\end{eqnarray*}
we have that
\[
G_\Half(z,w;0,\infty) = \frac{1}{c^2_*} \lim_{\epsilon,\delta
\rightarrow0^+}
\epsilon^{d-2}
\delta^{d-2} \Prob\{\xi< \chi< \infty\}.
\]

By the domain Markov property, and conformal invariance of $\mbox{SLE}$, one
can deduce that
\begin{eqnarray*}
M_t :\!& = & \E[G_\Half(z,w;0,\infty) | \F_t] \\
& = &G_{H_t}(z,w;0,\infty) \\
& = &|Z_t'(z)|^{2-d} |Z_t'(w)|^{2-d} \cdot G_\Half
(Z_t(z),Z_t(w);0,\infty)
\end{eqnarray*}
is a local martingale, where $Z_t$ is the unique conformal map defined
by (\ref{ZtDef}) which maps $H_t$ to $\Half$, sending
$\gamma
(t)$ to $0$. We will find the SDE which $M_t$ satisfies and use that
the drift must zero to find the differential equation that $G(x_1, y_1,
x_1, y_2) := G_\Half(x_1+iy_1,x_2+iy_2;0,\infty)$ must satisfy.

From (\ref{ZtDef}), we know that
\[
dZ_t(z) = \frac{a}{Z_t(z)} \,dt + d B_t ,
\]
and hence, letting $Z_t(z) = X_t(z) + iY_t(z)$, we see that
\begin{eqnarray*}
dX_t(z) & = &\frac{aX_t(z)}{X_t(z)^2 + Y_t(z)^2} \,dt + d B_t, \\
dY_T(z) & = &-\frac{aY_t(z)}{X_t(z)^2 + Y_t(z)^2} \,dt.
\end{eqnarray*}

To compute the SDE for $|Z_t'(z)|$, we must use the logarithm. First
note that
\[
dZ_t'(z) = -\frac{aZ_t'(z)}{Z_t(z)^2} \,dt\vadjust{\goodbreak}
\]
and hence that
\[
d [\log Z_t'(z)] = \frac{dZ_t'(z)}{Z_t'(z)} = -\frac{a}{Z_t(z)^2} \,dt.
\]
We may thus recover the norm of the absolute value by considering the
real part, yielding
\[
d |Z_t'(z)|^{2-d} = a(d-2)|Z_t'(z)|^{2-d}\frac{X_t(z)^2 -
Y_t(z)^2}{(X_t(z)^2+Y_t(z)^2)^2} \,dt.
\]
%

From these, we may compute the equation for $M_t$. Note that only
$X_t(z)$ and $X_t(w)$ have nonzero diffusion coefficients. Suppressing
the arguments of $G$ in the notation, we obtain the following:
\begin{eqnarray*}
dM_t & = & M_t\biggl[
a(d-2)\frac{X_t(z)^2 - Y_t(z)^2}{(X_t(z)^2+Y_t(z)^2)^2}+ a(d-2)\frac
{X_t(w)^2 - Y_t(w)^2}{(X_t(w)^2+Y_t(w)^2)^2} \\
&&\hphantom{M_t\biggl[}
{} + \frac{aX_t(z)}{X_t(z)^2 + Y_t(z)^2}\,\frac{\bd_{x_1}G}{G} + \frac
{aX_t(w)}{X_t(w)^2 + Y_t(w)^2}\,\frac{\bd_{x_2}G}{G} \\
&&\hphantom{M_t\biggl[}
{} -\frac{aY_t(z)}{X_t(z)^2 + Y_t(z)^2}\,\frac{\bd_{y_1}G}{G} -\frac
{aY_t(w)}{X_t(w)^2 + Y_t(w)^2}\,\frac{\bd_{y_2}G}{G} \\
&&\hspace*{133.2pt}
{} + \frac{1}{2}\,\frac{\bd_{x_1x_1}G}{G} + \frac{1}{2}\,\frac{\bd
_{x_2x_2}G}{G} + \frac{\bd_{x_1x_2}G}{G}
\biggr] \,dt\\
&&{} + M_t\biggl[\frac{\bd_{x_1}G}{G} + \frac{\bd_{x_2}G}{G}\biggr] \,dB_t.
\end{eqnarray*}
Collecting together the drift terms and specializing to $t=0$ yields
\begin{eqnarray*}
0 & = & a(d-2)\frac{x_1^2 - y_1^2}{(x_1^2+y_1^2)^2} G + a(d-2)\frac{x_2^2
- y_2^2}{(x_2^2+y_2^2)^2} G \\
&&{} + a\frac{x_1\,\bd_{x_1}G - y_1\,\bd_{y_1}G}{x_1^2 + y_1^2} + a\frac
{x_2\,\bd
_{x_2}G - y_2\,\bd_{y_2}G}{x_2^2 + y_2^2} \\
&&{} + \frac{1}{2}\,\bd_{x_1x_1}G + \frac{1}{2}\,\bd_{x_2x_2}G + \bd_{x_1x_2}G.
\end{eqnarray*}

This PDE has a particularly nice structure. Let
\[
L_i = a(d-2)\frac{x_i^2 - y_i^2}{(x_i^2+y_i^2)^2} + a\frac{x_i\,\bd_{x_i}
- y_i\,\bd_{y_i}}{x_i^2 + y_i^2} + \frac{1}{2}\,\bd_{x_ix_i}.
\]
This can be seen to be precisely the differential operator which arises
in the computation of the single point Green's function, but now we have
a copy for both $z$ and~$w$. With this we can rewrite the equation for
the multi-point Green's function as
\[
(L_1 + L_2 + \bd_{x_1x_2}) G = 0.
\]
Given this simple form it may be reasonable to look for solutions which
are, in some sense, asymptotically $G(z)G(w)$. Additionally, it is
worth noting that this extends to arbitrary $n$-point Green's functions by
\[
\Biggl[\sum_{i = 1}^n L_i + \sum_{1 \le i < j \le n} \bd_{x_ix_j}\Biggr] G = 0
\]
as one might expect.

The boundary conditions of this equation are not clear, and their
determination may provide bounds of intrinsic interest.

The above equation shows the PDE in its most symmetric form; however,
in order to find an explicit solution, it may be useful to exploit the
scaling rule for the Green's function to reduce this to an equation for a
function of three real variables. There is no unique way to do so, and
no such reductions have lead to a particularly simple equation. A
reasonable example would be to scale the above equation so that $y_1=1$
in which case we can find a three real variable function $\hat G$ so that
\[
G(x_1,y_1,x_2,y_2) = y_1^{2(d-2)}\hat G\biggl(\frac{x_1}{y_1},\frac
{x_2}{y_1},\frac{y_2}{y_1}\biggr).
\]
From this the PDE can be derived; however, the result is not illuminating.
%
\end{appendix}

\section*{Acknowledgments}

We would like to thank the anonymous reviewer of this paper for a very
careful reading and for pointing out a number of errors in an earlier
draft of this paper. We would also like to thank Mohammad Rezaei for
helpful comments.


%

\printaddresses

\end{document}